\documentclass[12pt,reqno]{amsart}
\usepackage{graphicx}
\usepackage{amssymb,amsmath}
\usepackage{amsthm}
\usepackage{color}
\usepackage[pdf]{pstricks}
\usepackage{hyperref}

\numberwithin{equation}{section}
\numberwithin{figure}{section}

{\begin{trivlist} \item[]{{\em Proof} }}%
	{\hspace*{\fill}$\rule{.3\baselineskip}{.35\baselineskip}$\end{trivlist}}

\newcommand{\beq}{\begin{equation}}
	\newcommand{\eeq}{\end{equation}}

\newtheorem{theorem}{Theorem}[section]

\newtheorem{lemma}{Lemma}[section]

\newtheorem{proposition}{Proposition}[section]
\newtheorem{remark}{Remark}[section]
\newtheorem*{lemma*}{Lemma}

\begin{document}
	
	\title[Gross-Pitaevskii equation in the energy-critical case]{Ground state of the Gross-Pitaevskii equation with a harmonic potential in the energy-critical case}
	
	\author{Dmitry E. Pelinovsky}
	\address[D.E. Pelinovsky]{Department of Mathematics and Statistics, McMaster University,
		Hamilton, Ontario, Canada, L8S 4K1}
	\email{dmpeli@math.mcmaster.ca}
	
	\author{Szymon Sobieszek}
	\address[S. Sobieszek]{Department of Mathematics and Statistics, McMaster University,	Hamilton, Ontario, Canada, L8S 4K1}
	\email{sobieszs@mcmaster.ca}

	
	\begin{abstract}
		Ground state of the energy-critical Gross–Pitaevskii equation with a harmonic potential can be constructed variationally. It exists in a finite interval of the eigenvalue parameter. The supremum norm of the ground state vanishes at one end of this interval and diverges to  infinity at the other end. We explore the shooting method in the limit of large norm to prove that the ground state is pointwise close to the Aubin--Talenti solution of the energy-critical wave equation in near field and to the confluent hypergeometric function in far field. The shooting method gives the precise dependence of the eigenvalue parameter versus the supremum norm.	
	\end{abstract}
	
	\date{\today}
	\maketitle
	
	
	\section{Introduction}
	
	We consider the stationary Gross--Pitaevskii equation with a harmonic potential,
	\begin{equation}
		\label{GP}
		(-\Delta + |x|^2) u - |u|^{2p} u = \lambda u, 
	\end{equation}
	where $x \in \mathbb{R}^d$, $\lambda \in \mathbb{R}$, and $u \in \mathbb{R}$. Existence of its ground state (a positive and radially decreasing solution) has been addressed by using variational methods in the energy subcritical \cite{Fuk,KW} and critical \cite{Selem2011,SK2012} cases, where the critical exponent is $p = \frac{2}{d-2}$ if $d \geq 3$. Uniqueness of the ground state 
	was proven in the energy subcritical \cite{HO1,HO2} and critical 
	\cite{SW} cases. 
	
	Variational methods are not applicable in the energy supercritical case $p(d-2) > 2$ if $d \geq 3$ for which a more efficient shooting method was developed in 
	our previous work \cite{BIZON2021112358,Pel-Sob} (see also \cite{Ficek} for study of the Schr\"{o}dinger--Newton--Hooke model). The shooting method is based on the reformulation of the existence problem after the Emden--Fowler transformation \cite{F,JL} and construction of two analytic families 
	of solutions, one family gives a bounded solution near $r = 0$ with parameter 
	$b := u(0) \equiv \| u \|_{L^{\infty}}$ and the other family
	gives a decaying solution as $r \to \infty$ with parameter 
\begin{equation}
\label{c-asymptotics}
	c := \lim_{r \to \infty} u(r) e^{\frac{1}{2} r^2} r^{\frac{d-\lambda}{2}}.
\end{equation}
The shooting method gives robust results in the large-norm limit as $b \to \infty$. In an asymptotic region, where both solution families coexist, the matching condition gives a condition on $(c,\lambda)$ as a function of $b$ 
which determines the solution curve $\lambda = \lambda(b)$. The $c$-family of solutions is defined in a local neighborhood of the limiting singular solution constructed in \cite{Selem2013} for $\lambda = \lambda_{\infty} \in (0,d)$ such that $\lambda(b) \to \lambda_{\infty}$ as $b \to \infty$ along the solution curve. As follows from the shooting method  \cite{BIZON2021112358} under some non-degeneracy assumptions, the convergence is oscillatory for $2 + \frac{2}{p} < d < d_*(p)$ and monotone for $d > d_*(p)$, where 
$$
d_*(p) := 6 + \frac{2}{p} + 2 \sqrt{4 + \frac{2}{p}}.
$$
The same shooting method was extended in \cite{Pel-Sob} to compute the Morse index of the ground state in the monotone case from the Morse index of the limiting singular solution. It was recently proven in \cite{PWW} by using comparison with the stationary Schr\"{o}dinger equation solvable in terms of  the confluent hypergeometric functions (see \cite{AS1972} for review) that the Morse index of the limiting singular solution is infinite 
	in the oscillatory case and is equal to one in the monotone case 
	for large values of $d > d_*(p)$. 
	
	Properties of the energy-supercritical Gross--Pitaevskii equation with a harmonic potential are very similar to those for the energy-supercritical nonlinear Schr\"{o}dinger equation in a ball. See \cite{Budd1989,Budd_Norbury1987,DF} for the developments in the shooting method, \cite{MP} for convergence to the limiting singular solution, 
	and \cite{GuoWei,KikuchiWei} for computation of the Morse index.
	
	The purpose of this work is to extend the shooting method to the energy-critical case and to obtain the asymptotic representation of $\lambda(b)$ as $b \to \infty$. As far as we are aware, the shooting method has not been previously developed in the context of the energy-critical case, for which the variational approximations are more common.
	
	For the shooting method in the energy-critical case with $d = 2 + \frac{2}{p}$, we introduce the same two analytic families of solutions, 
	the $b$-family is defined by parameter $b := u(0)$ and the $c$-family is defined by parameter $c$ in the asymptotic behavior (\ref{c-asymptotics}). Contrary to the energy-supercritical case, the $c$-family exists in a local neighborhood of a spatially decaying solution to the stationary Schr\"{o}dinger equation
	\begin{equation*}
		V''(r) + \frac{d-1}{r} V'(r) - r^2 V(r) + \lambda V(r) = 0, 
	\end{equation*}
	which is satisfied by $V(r) = c e^{-\frac{1}{2} r^2} \mathfrak{U}(r^2;\alpha,\beta)$, where $c \in \mathbb{R}$ is arbitrary and $\mathfrak{U}(z;\alpha,\beta)$ is the Tricomi function (see \cite{AS1972}) with 
	\begin{equation}
		\label{alpha-beta}
		\alpha := \frac{p+1}{2p}-\frac{\lambda}{4}, \qquad 
		\beta := 1+\frac{1}{p}. 
	\end{equation}
Furthermore, contrary to the energy-supercritical case, the $b$-family exists in a local neighborhood of the algebraic soliton 
	\begin{eqnarray}
		\label{alg-soliton}
		U_b(r) = \frac{b}{(1 + \alpha_p b^{2p} r^2)^{\frac{1}{p}}},  \qquad \alpha_p := \frac{p^2}{4(1+p)},
	\end{eqnarray}
	where the parameter $b$ has been introduced from the condition $b = U_b(0)$.
	The algebraic soliton (also called the Aubin-Talenti solution \cite{Aubin,Talenti}) satisfies the nonlinear wave equation $-\Delta U_b = U_b^{2p+1}$ for every $b > 0$. It has been used in many studies of the energy-critical wave equations in bounded domains as in the pioneering work \cite{BN83} and in follow-up works \cite{FKK20,Gustaf,H91,HV01,R90}.  
	In the context of the stationary Gross--Pitaevskii equation (\ref{GP}), it was used in \cite{Selem2011} in order to obtain the lower bound 
	on the dependence $\lambda(b)$ from a variational method. Recently in \cite{PWW}, the variational methods and the elliptic estimates were extended in order to get the upper bound on the dependence $\lambda(b)$. We will use the shooting method to justify the relevance of the algebraic soliton (\ref{alg-soliton}) for the asymptotic behavior of $\lambda(b)$ as $b \to \infty$.
	
			\begin{figure}[hbt!]
		\centering  
		\includegraphics[width=7cm,height=6cm]{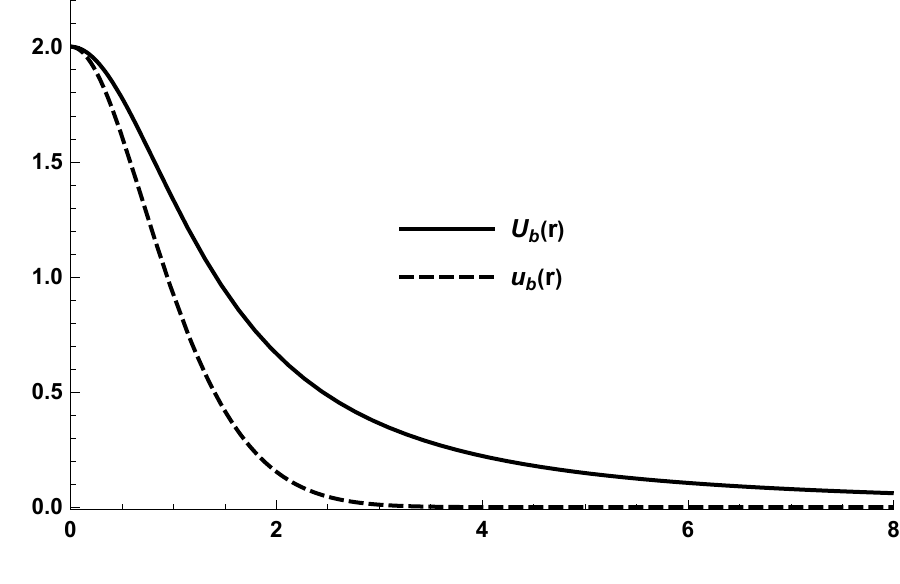}   
		\includegraphics[width=7cm,height=6cm]{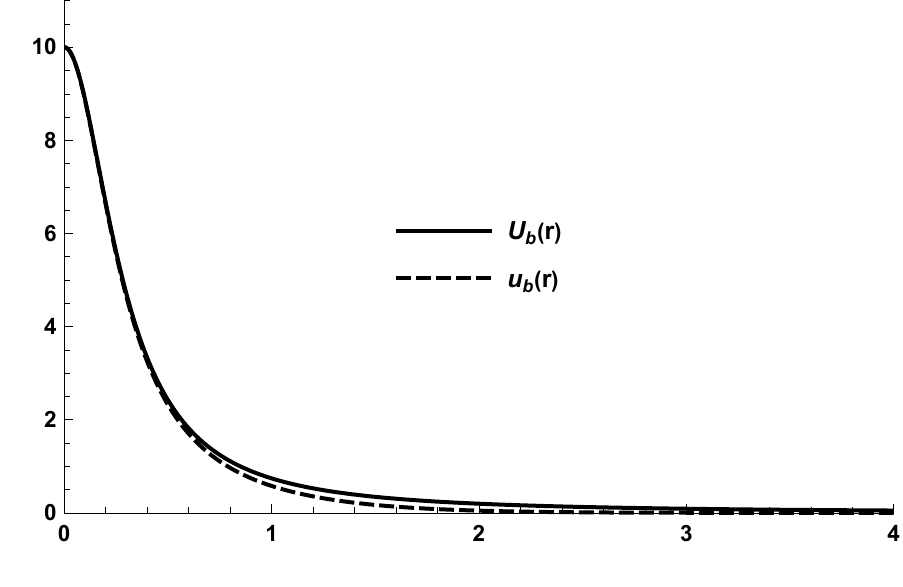}  
		\caption{Ground states of the stationary equation \eqref{GP} with $p=1$ and $d=4$ compared with the algebraic soliton \eqref{alg-soliton} for $b = 2$ (left) and $b = 10$ (right).} 
		\label{fig-1}
	\end{figure}
	
	\begin{figure}[hbt!]
		\centering  
		\includegraphics[width=10cm,height=7cm]{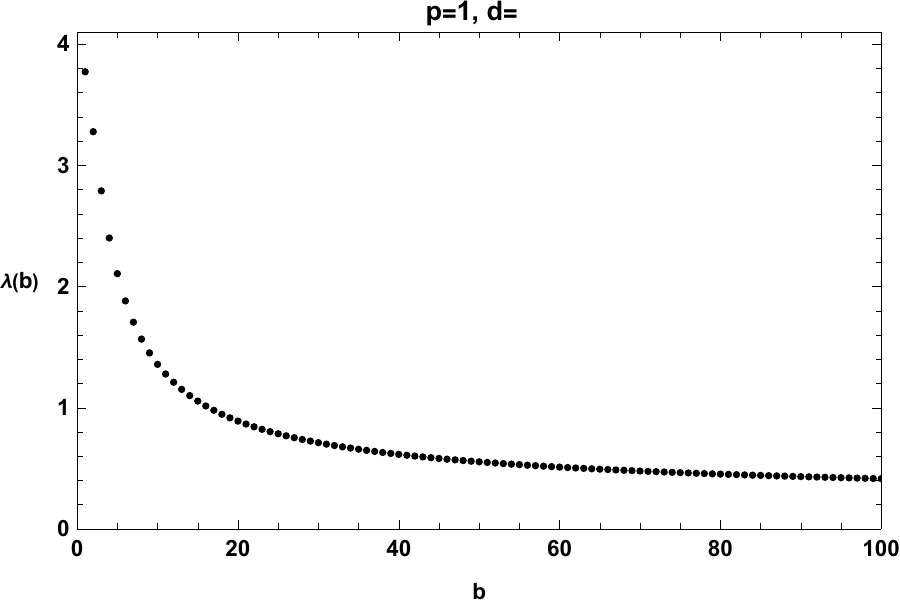}   
		\caption{The dependence $\lambda = \lambda(b)$ for $p=1$ and $d=4$.} 
		\label{fig-11}
	\end{figure}
	
	Figure \ref{fig-1} shows the numerically obtained profile $u_b$ versus $r$ 
	in comparison with the profile $U_b$ for $b = 2$ (left) and $b = 10$ (right). Visualization is given for $p = 1$ (that corresponds to $d = 4$). Results for other values of $p \in (0,1)$ are similar. The two profiles are different for $b = 2$ but the discrepancy gets smaller 
	for $b = 10$ and becomes invisible for larger values of $b$. The values of $\lambda$ are uniquely defined in terms of $b$ along the curve $\lambda = \lambda(b)$ which is shown in Figure \ref{fig-11} for $p = 1$.

The following theorem presents outcomes of the shooting method, which is the main result of this work. We use the following notations: 
\begin{itemize}
	\item $\lambda(b) \sim \lambda_0(b)$ denotes the asymptotic equivalence
	in the sense $\lim\limits_{b \to \infty} \lambda_0(b)^{-1} \lambda(b) = 1$,
	\item $\lambda(b) = \mathcal{O}(b^q)$ denotes the order of magnitude in the sense that 
	$|\lambda(b)| \leq C b^q$ for some $C > 0$ and all sufficiently large $b$.
\end{itemize}
	
	\begin{theorem}
		\label{theorem-main}
Fix $p=\frac{2}{d-2} \in (0,1)$ for $d > 4$ and let $\lambda = \lambda(b)$ be the solution curve for the ground state $u = u_b$ of the stationary Gross--Pitaevskii equation \eqref{GP} satisfying $u_b(0) = b$, $u_b'(r) < 0$ for $r > 0$, and $u_b(r) \to 0$ as $r \to \infty$. Then,
		\begin{equation}
			\label{lambda-b}
			\lambda(b) \sim
			C_p \left\{ \begin{array}{ll} \displaystyle 
				b^{-2(1-p)}, & \quad \frac{1}{2} < p < 1, \\ \displaystyle 
				b^{-1} \log b, & \quad p = \frac{1}{2}, \\ \displaystyle 
				b^{-2p}, & \quad 0 < p < \frac{1}{2}, \\
			\end{array} \right. \quad \mbox{\rm as} \quad b \to \infty,
		\end{equation}
with 
$$
C_p = \left\{ \begin{array}{ll}
\displaystyle 
-\frac{\Gamma\left(\frac{p+1}{2p}\right)\Gamma\left(-\frac{1}{p}\right)\Gamma\left(\frac{2}{p}\right)}{(1+p) \Gamma\left(\frac{p-1}{2p}\right)\Gamma\left(\frac{1}{p}\right)\Gamma\left(\frac{1}{p}-1\right) \Gamma\left(\frac{1}{p}+1\right)} 
\left[\frac{4 (1+p)}{p^2}\right]^{\frac{1}{p}}, & \quad \frac{1}{2} < p < 1, \\ \displaystyle 
144, & \quad p = \frac{1}{2}, \\
\displaystyle 
\frac{8(1+p)^2}{p^2 (1-2p)}, & \quad 0 < p < \frac{1}{2}. \\
\end{array} \right.
$$
Moreover, for every $a \in (0,\frac{p}{1+p})$, there exist $B_a,C_a > 0$
		such that for every $b \geq B_a$, we have 
		\begin{equation}
			\label{asymptotics-u-b}
			\sup_{r \in [0,b^{-p(1-a)}]} b^{-1} \left| u_b(r) -  U_b(r)\right| \leq C_a b^{-2p(1-a)}, 
		\end{equation}
		\begin{equation}
			\label{asymptotics-u-c}
			\sup_{r \in [b^{-p(1-a)},1]} r^{\frac{2}{p}} \left| u_b(r) - c(b) e^{-\frac{1}{2} r^2} \mathfrak{U}(r^2;\alpha,\beta) \right| \leq C_a |c(b)|^{2p+1} b^{2p(1-a)}
		\end{equation}
		and 
\begin{equation}
		\label{asymptotics-u-d}
		\sup_{r \in [1,\infty)} \left| e^{\frac{1}{2} r^2} r^{\frac{d-\lambda(b)}{2}} u_b(r) - c(b) r^{\frac{d-\lambda(b)}{2}} \mathfrak{U}(r^2;\alpha,\beta) \right| \leq C_a |c(b)|^{2p+1},
		\end{equation}
		where $c = c(b) \sim A_p b^{-1}$ as $b \to \infty$ for some $A_p > 0$.
	\end{theorem}
	
	\begin{remark}
		The asymptotic result (\ref{lambda-b}) coincides with Theorem 1.1 in \cite{PWW} obtained by the variational theory and elliptic estimates. It follows from Remark 1.2 in \cite{PWW} that 
		there exists $C_d$ such that 
		\begin{equation}
			\label{lambda-b-PWW}
			\lambda(\varepsilon) \sim C_d
			\left\{ \begin{array}{ll}
				\varepsilon & \quad d = 5, \\
				\varepsilon^2 |\log \varepsilon| & \quad d = 6, \\
				\varepsilon^2 & \quad d \geq 7, \\
			\end{array} \right.
		\end{equation}
		where $\varepsilon = b^{-p}$ is defined from the algebraic soliton (\ref{alg-soliton}). The case  $d \geq 7$ corresponds to 
		$0 < p < \frac{1}{2}$ as in (\ref{lambda-b}). For $d = 6$, we have $p = \frac{1}{2}$ so that $\varepsilon^2 |\log \varepsilon| \sim b^{-1} \log b$ as in (\ref{lambda-b}). For $d = 5$, we have $p = \frac{2}{3}$ so that $\varepsilon \sim b^{-2/3} = b^{-2(1-p)}$ as in (\ref{lambda-b}).
	\end{remark}

\begin{remark}
	\label{remark-fail}
It also follows from Remark 1.2 in \cite{PWW} that $\lambda(\varepsilon) \sim 
C |\log \varepsilon|^{-1}$ for $d = 4$ and $\lambda(\varepsilon) - 1\sim C \varepsilon$ for $d = 3$. In our notations with $\varepsilon = b^{-p}$, this would correspond to $\lambda(b) \sim C (\log b)^{-1}$ for $p = 1$ 
and $\lambda(b) - 1 \sim C b^{-2}$ for $p = 2$. However, we have found that the shooting method based on the $b$-family and the $c$-family can be applied for $p \in (0,1)$ but needs some further modifications for $p \geq 1$.
\end{remark}

\begin{remark}
	Since $\mathfrak{U}(r^2;\alpha,\beta) = \mathcal{O}(r^{-\frac{2}{p}})$ as $r \to b^{-p(1-a)}$, bound (\ref{asymptotics-u-c}) shows that $u_b(r) = \mathcal{O}(b^{1-2a})$ as $r \to b^{-p(1-a)}$. This is smaller than $u_b(r) = \mathcal{O}(b)$ as $r \to 0$ in the bound (\ref{asymptotics-u-b}). Since $\mathfrak{U}(r^2;\alpha,\beta) = \mathcal{O}(r^{-\frac{d-\lambda(b)}{p}})$ as $r \to \infty$, bound (\ref{asymptotics-u-d}) shows that $u_b(r)$ satisfies the asymptotic behavior (\ref{c-asymptotics}) with $c = c(b) = \mathcal{O}(b^{-1})$ as $b \to \infty$. 
\end{remark}
	
	Figure \ref{fig-2} illustrates relevance of the asymptotic result (\ref{lambda-b}) for the solution curve  $\lambda=\lambda(b)$. For a given dimension $d$ and the critical exponent $p=\frac{2}{d-2}$, we numerically find $\lambda(b)$ and plot it versus $b$ in comparison with the asymptotic dependence (\ref{lambda-b}). The left and right panels show the plots for $d = 7$ when $p = \frac{2}{5}$ and $\lambda(b) \sim C_p b^{-4/5}$ and for $d = 5$ when $p = \frac{2}{3}$ and $\lambda(b) \sim C_p b^{-2/3}$, where $C_p$ is obtained from the best least square fit. The proximity between the numerical and analytical curves becomes obvious in the log-log plot for larger values of $b$. 
	
	\begin{figure}[hbt!]
		\centering
		\includegraphics[width=7cm,height=6cm]{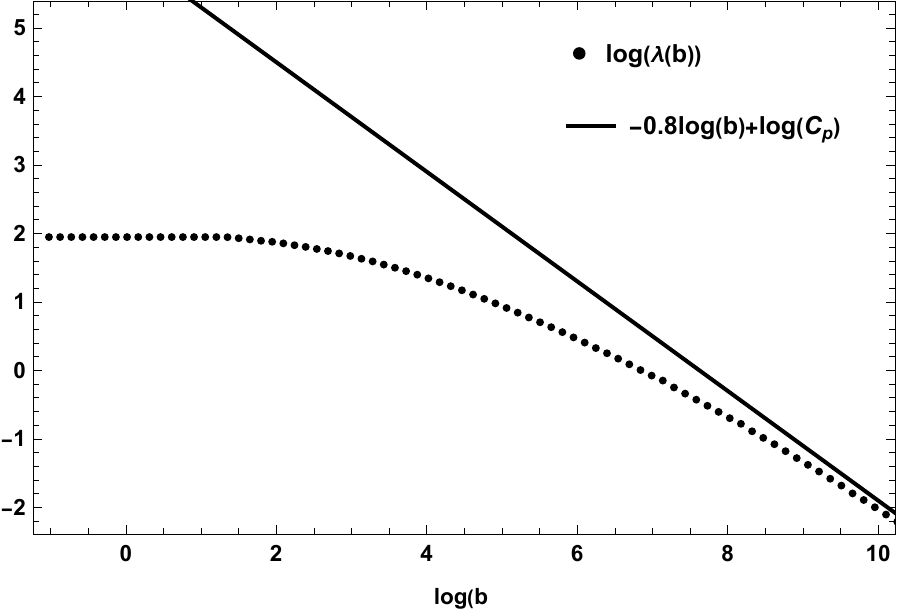}\qquad
		\includegraphics[width=7cm,height=6cm]{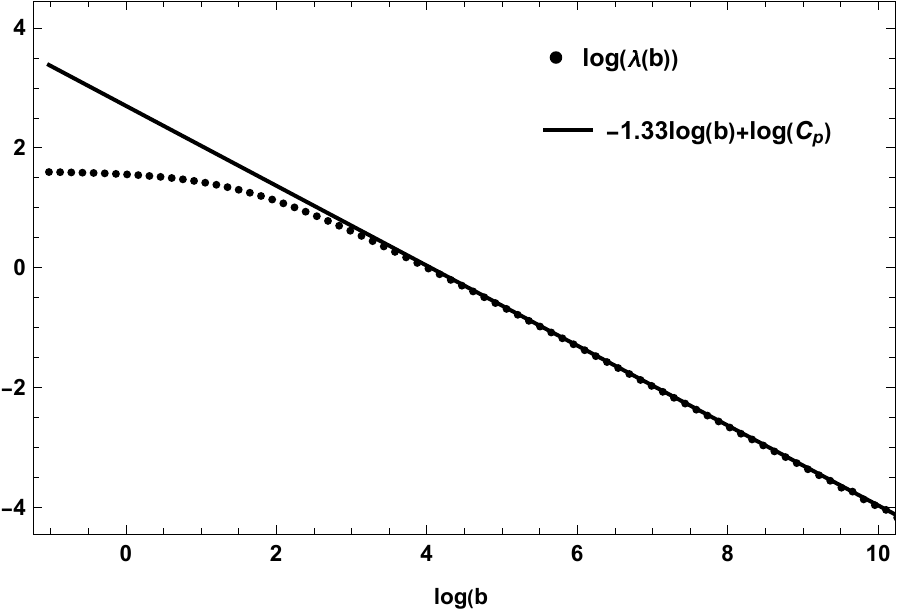}
		\caption{Log-log graphs of $\lambda(b)$ versus $b$ for the ground state of the stationary equation \eqref{GP} for $d=7$ (left) and $d=5$ (right) compared with the analytical dependence given in \eqref{lambda-b}.}
		\label{fig-2}
	\end{figure}

	Our strategy to prove Theorem \ref{theorem-main} is as follows. 
	Section \ref{sec-2} contains preliminary results where 
	the existence problem is reformulated after the Emden--Fowler transformation 
	and the two solution families and their truncated limits are clearly identified. 
	Section \ref{sec-3} gives analysis of the $b$-family in a local 
	neighborhood of the algebraic soliton (\ref{alg-soliton}) which becomes 
	the exponentially decaying soliton after the Emden--Fowler transformation. 
	Section \ref{sec-4} describes analysis of the $c$-family 
	in a local neighborhood of the confluent hypergeometric functions. 
	Theorem \ref{theorem-main} is proven in Section \ref{sec-5} where the two families are considered in the common asymptotic region with parameters $c$ and $\lambda$ obtained uniquely in the asymptotic limit $b \to \infty$. 
	Besides the asymptotic dependence (\ref{lambda-b}) which recovers independently the result (\ref{lambda-b-PWW}) obtained in \cite{PWW} with different methods, the main outcome of this work is the precise asymptotic construction 
	of the ground state with pointwise estimates (\ref{asymptotics-u-b}), (\ref{asymptotics-u-c}), and (\ref{asymptotics-u-d})  near the Aubin-Talenti solution and the confluent hypergeometric function.
	
	\section{Preliminary results}
	\label{sec-2}
	
	As in our previous works \cite{BIZON2021112358,Pel-Sob}, we reformulate the existence problem for the ground state of the stationary Gross--Pitaevskii equation (\ref{GP}) as the following initial-value problem:
	\begin{equation}
		\label{eq:f_b_ivp}
		\left\{ \begin{array}{ll}
			f''(r) + \frac{d-1}{r} f'(r) - r^2 f(r) + \lambda f(r) + |f(r)|^{2p} f(r) =0, & \quad r > 0,  \\
			f(0)=b, \quad f'(0)=0, & \quad 
		\end{array} \right.
	\end{equation}
	where $b > 0$ is the free parameter and $d = 2+\frac{2}{p}$ is defined in terms of $p > 0$ in the energy-critical case. We say that the solution 
	of the initial-value problem (\ref{eq:f_b_ivp}) is a ground state 
	if $f'(r) < 0$ for $r > 0$ and $f(r) \to 0$ as $r \to \infty$. 
	Similarly to Lemmas 3.2 and 3.4 in \cite{BIZON2021112358} obtained 
	in the particular case $p = 1$, the existence of a unique classical 
	solution to the initial-value problem \eqref{eq:f_b_ivp} can be concluded 
	by using the integral equation formulation and the Lyapunov function method. 
	We skip the proof since it is standard and state that for every $p > 0$, $\lambda \in \mathbb{R}$, and $b>0$, there exists a unique classical solution $f \in C^2(0,\infty)$ to the initial-value problem \eqref{eq:f_b_ivp} satisfying the asymptotic behavior:
	\begin{equation}
		f(r) = b - \frac{p (\lambda+ b^{2p})}{4(p+1)} b r^2 + \mathcal{O}(r^4), \quad \text{as } \;\; r \to 0.
		\label{eq:f_b_asympt}
	\end{equation} 
	
	The singularity of the stationary equation \eqref{eq:f_b_ivp} at $r=0$ is unfolded by introducing the Emden-Fowler transformation:
	\begin{equation}
		r = e^{t}, \quad \Psi(t) = e^{\frac{t}{p}}f(e^{t}).
		\label{eq:EF_transf}
	\end{equation}
	After the transformation of variables, $\Psi$ satisfies the second-order nonautonomous equation
	\begin{equation}
		\Psi''(t) -\frac{1}{p^2}\Psi(t) + |\Psi(t)|^{2p} \Psi(t) = -\lambda e^{2t}\Psi(t) + e^{4t}\Psi(t).
		\label{eq:Psi_eq}
	\end{equation}
	
	We say that the $b$-family of solutions to Eq. (\ref{eq:Psi_eq}) is defined 
	by applying the transformation (\ref{eq:EF_transf}) to the unique solution of the initial-value problem (\ref{eq:f_b_ivp}). The corresponding $b$-solution, denoted as $\Psi_b(t)$, satisfies the asymptotic behaviour 
	\begin{equation}
		\Psi_b(t) = be^{\frac{t}{p}} \left[ 1 -   \frac{p (\lambda+ b^{2p})}{4(p+1)} e^{2 t} + \mathcal{O}(e^{4 t}) \right], \quad \text{ as } \;\; t\to-\infty,
		\label{eq:Psi_b_asympt}
	\end{equation}
	which follows from \eqref{eq:f_b_asympt}. Thus, the $b$-family of solutions decays to zero as $t \to -\infty$. We will show in Section \ref{sec-3} that the $b$-family stays close to the positive homoclinic orbit of the truncated version of  Eq. \eqref{eq:Psi_eq} given by the second-order autonomous equation
	\begin{equation}
		\Theta''(t)-\frac{1}{p^2}\Theta(t)+|\Theta(t)|^{2p} \Theta(t)=0.
		\label{eq:theta_nonlinear}
	\end{equation}
	The second-order equation (\ref{eq:theta_nonlinear}) is integrable with the first-order invariant
	\begin{equation}
		\label{eq:theta-first}
		\frac{1}{2} (\Theta')^2 - \frac{1}{2p^2} \Theta^2 + \frac{1}{2(p+1)} \Theta^{2(p+1)} = E,
	\end{equation}
	where $E$ is constant along the classical solutions of Eq. (\ref{eq:theta_nonlinear}). The origin in the $(\Theta,\Theta')$-plane is a saddle point. The unique (up to translation) positive homoclinic orbit exists at the energy level $E = 0$ for every $p > 0$. The homoclinic orbit can be found explicitly in the form
	\begin{equation}
		\Theta_h(t+t_0) = \frac{e^{\frac{t+t_0}{p}}}{(1+\alpha_p e^{2(t+t_0)})^{\frac{1}{p}}}, \qquad \alpha_p:=\frac{p^2}{4\left(1+p\right)},
		\label{eq:Theta_alg_sol}
	\end{equation}
	where $t_0\in\mathbb{R}$ is an arbitrary parameter of translation. Since
	\begin{equation}
 \Theta_h(t) =	\left\{ \begin{array}{ll} e^{\frac{t}{p}} [1+\mathcal{O}(e^{2t})] \quad & \text{ as } \;\;t\to-\infty, \\ 
\alpha_p^{-\frac{1}{p}} e^{-\frac{t}{p}} [1+\mathcal{O}(e^{-2t})] \quad & \text{ as } \;\; t\to+\infty, \end{array} \right.
		\label{eq:Theta_asympt}
	\end{equation}
it follows by comparison with (\ref{eq:Psi_b_asympt}) that 
	$\Psi_b(t) \sim b \Theta_h(t)$ as $t \to -\infty$ for which 
	the translation parameter $t_0$ in (\ref{eq:Theta_alg_sol}) is uniquely selected 
	as $t_0 = p \log b$. With this choice for $t_0$, 
	we observe that $\Theta_h(t+p \log b)$ after transformation (\ref{eq:EF_transf}) coincides with the algebraic soliton $U_b(r)$ given by (\ref{alg-soliton}).
	
	Next, we introduce another analytical family of solutions to Eq. \eqref{eq:Psi_eq} that decay to zero as $t\to +\infty$, which we call 
	the $c$-family and denote as $\Psi_c(t)$. We show in Section \ref{sec-4} 
	that the $c$-family stays close to the decaying solutions 
	of the linearized version of Eq. (\ref{eq:Psi_eq}) given by the linear 
	second-order nonautonomous equation 
	\begin{equation}
		\Upsilon''(t) -\frac{1}{p^2}\Upsilon(t) + \lambda e^{2t} \Upsilon(t) - e^{4t}\Upsilon(t) = 0.
		\label{eq:Upsilon_eq}
	\end{equation}
	By using the change of variables
	\begin{equation}
		\label{change-variables}
		z = e^{2t}, \quad \Upsilon(t) = z^{\frac{1}{2p}} e^{-\frac{1}{2} z} u(z),
	\end{equation}
	the second-order equation (\ref{eq:Upsilon_eq}) becomes the confluent hypergeometric equation (also known as the Kummer equation):
	\begin{equation}
		zu''(z) + \left(\beta-z\right)u'(z) - \alpha u(z) = 0,
		\label{eq:kummer}
	\end{equation}
	with parameters $\alpha$ and $\beta$ given by (\ref{alpha-beta}).
	Two special solutions of the Kummer equation \eqref{eq:kummer} 
	are given by the Kummer function $\mathfrak{M}(z;\alpha,\beta)$ and the Tricomi function 
	$\mathfrak{U}(z;\alpha,\beta)$, which are defined as follows \cite{AS1972}. The Kummer function is defined by the power series 
	\begin{align}
		\label{Kummer}
		\mathfrak{M}(z;\alpha,\beta) &= 
		\sum_{k=0}^{\infty} \frac{(\alpha)_k}{(\beta)_k} \; \frac{z^k}{k!} \\
		&=
		1 + \frac{\alpha}{\beta}\frac{z}{1!} + \frac{\alpha(\alpha+1)}{\beta(\beta+1)}\frac{z^2}{2!} + \frac{\alpha(\alpha+1)(\alpha+2)}{\beta(\beta+1)(\beta+2)}\frac{z^3}{3!} + \ldots, \notag
	\end{align}
	hence it is bounded as $z \to 0$. The Tricomi function satisfies 
	the asymptotic behavior 
	\begin{equation}
		\label{Tricomi}
		\mathfrak{U}(z;\alpha,\beta) \sim z^{-\alpha} \left[ 1 + \mathcal{O}(z^{-1}) \right] \quad \mbox{\rm as} \;\; z \to +\infty,
	\end{equation}
	hence it is decaying as $z \to +\infty$ if $\alpha > 0$. In fact, for every $\lambda \in (-\infty,d)$, we have
	\begin{equation}
	\label{eq:alph}
	\alpha = \frac{p+1}{2p} - \frac{\lambda}{4} > \frac{p+1}{2p} -\frac{d}{4} = \frac{p+1}{2p}-\frac{1}{2}-\frac{1}{2p}=0,
	\end{equation}
	so that $\alpha > 0$ is satisfied in the energy-critical case.
	By \cite[13.1.3]{AS1972}, the Tricomi function can be represented in the superposition form
	\begin{equation}
		\label{scatering-1}
		\mathfrak{U}(z;\alpha,\beta) = \frac{\pi}{\sin \pi \beta}\left[\frac{\mathfrak{M}(z;\alpha,\beta)}{\Gamma(1+\alpha-\beta)\Gamma(\beta)} -z^{1-\beta}\frac{\mathfrak{M}(z;1+\alpha-\beta,2-\beta)}{\Gamma(\alpha)\Gamma(2-\beta)} \right],
	\end{equation}
which is true for $\beta \notin \mathbb{Z}$ but can also be used in the limit 
$\beta \to \mathbb{Z}$. By using the identity 
\begin{equation}
\label{Gamma-function}
\frac{\pi}{\sin\pi z}=\Gamma(1-z)\Gamma(z), \quad z \notin \mathbb{Z}.
\end{equation}
we can rewrite \eqref{scatering-1} for $\beta \notin \mathbb{N}$ as
	\begin{equation}
		\label{scatering-1_simple}
		\mathfrak{U}(z;\alpha,\beta) = \frac{\Gamma(1-\beta)}{\Gamma(1+\alpha-\beta)}\mathfrak{M}(z;\alpha,\beta) +z^{1-\beta}\frac{\Gamma(\beta-1)}{\Gamma(\alpha)}\mathfrak{M}(z;1+\alpha-\beta,2-\beta).
	\end{equation}
	By \cite[13.1.6]{AS1972}, if $\beta = n + 1 \in \mathbb{N}$, then 
	\begin{align}
		\mathfrak{U}(z;\alpha,n+1) &= \frac{(-1)^{n+1}}{n! \Gamma(\alpha-n)} ( \mathfrak{M}(z;\alpha,n+1) \log z  \nonumber \\
		& \qquad + \sum_{k=0}^{\infty} \frac{(\alpha)_k}{(n+1)_k} \; \frac{z^k}{k!} [\psi(\alpha + k) - \psi(1+k) - \psi(1+n+k)] ) \nonumber \\
			& + \frac{1}{\Gamma(\alpha)}  
			\sum_{k=1}^n \frac{(k-1)! (1-\alpha+k)_{n-k}}{(n-k)!} z^{-k}, 
		\label{scatering-2}
	\end{align}
	where $\psi(z) = \Gamma'(z)/\Gamma(z)$.
	
	By means of the transformation (\ref{change-variables}), Tricomi function 
	determines a suitable solution of the linear equation (\ref{eq:Upsilon_eq}):
	\begin{equation}
		\label{eq:Ups_h}
		\Upsilon_h(t) = e^{\frac{t}{p}} e^{-\frac{1}{2} e^{2t}} \mathfrak{U}(e^{2t};\alpha,\beta).
	\end{equation}
	This solution is considered to be the leading-order approximation of the $c$-family such that 
	$\Psi_c(t) \sim c \Upsilon_h(t)$ as $t \to +\infty$ satisfies the asymptotic behavior 
	\begin{equation}
		\label{far-field}
		\Psi_c(t) \sim c e^{-\frac{(2-\lambda) t}{2}} e^{-\frac{1}{2} e^{2t}} \quad \mbox{\rm as} \;\; t \to +\infty.
	\end{equation}
	The ground state of Theorem \ref{theorem-main}
	is the connection of the unique solution of the initial-value problem (\ref{eq:f_b_ivp}) satisfying (\ref{eq:f_b_asympt}) with the unique 
	solution satisfying the decay behavior
	\begin{equation}
		\label{far-field-r}
		f(r) \sim c r^{-(1 + \frac{1}{p} - \frac{\lambda}{2})} e^{-\frac{1}{2} r^2} \quad \mbox{\rm as} \;\; r \to \infty,
	\end{equation}
	which follows from (\ref{eq:EF_transf}) and (\ref{far-field}).
	The connection between (\ref{eq:f_b_asympt}) and (\ref{far-field-r}) only exists for some specific values of $c = c(b)$ and $\lambda = \lambda(b)$. Thus, the main question is to find and to justify the analytical expressions 
	for $c(b)$ and $\lambda(b)$ in the asymptotic limit $b \to \infty$.

	\section{Persistence of the $b$-family of solutions}
	\label{sec-3}
	
	The $b$-family of solutions $\Psi_b$ of the differential equation (\ref{eq:Psi_eq}) satisfying (\ref{eq:Psi_b_asympt}) is considered in 
	a neighborhood of the homoclinic orbit $\Theta_h$ of the differential 
	equation (\ref{eq:theta_nonlinear}) satisfying (\ref{eq:Theta_asympt}). 
	Since the comparison gives $\Psi_b(t) \sim b \Theta_h(t) \sim \Theta_h(t + p \log b)$ as $t \to -\infty$, we translate $\Psi_b(t)$ by $-p\log b$ and  introduce the perturbation term
	$$
	\gamma(t):=\Psi_b(t-p\log b) - \Theta_h(t),
	$$  
	which satisfies
	\begin{equation}
		\mathcal{L}\gamma = f_b(\Theta_h+\gamma) - N(\Theta_h,\gamma),
		\label{eq:L_gam}
	\end{equation}
	where $f_b(t) := -\lambda b^{-2p} e^{2t} + b^{-4p}e^{4t}$, 
	\begin{align*}
		(\mathcal{L}\gamma)(t) := \gamma''(t) -\frac{1}{p^2}\gamma(t) + (2p+1) |\Theta_h(t)|^{2p}\gamma(t), 
	\end{align*}
	and 
	\begin{align*}
		N(\Theta_h,\gamma) := |\Theta_h+\gamma|^{2p} (\Theta_h + \gamma) - |\Theta_h|^{2p} \Theta_h - (2p+1) |\Theta_h|^{2p}\gamma.
	\end{align*}
	
	\begin{remark}
		Since $\Theta_h(t)$ is positive for all $t \in \mathbb{R}$ 
		and $\Theta_h(t) + \gamma(t)$ is shown to be positive in the region of $t$ where 
		we analyze persistence of the $b$-family of solutions, we can neglect 
		writing modulus signs in $\mathcal{L} \gamma$ and $N(\Theta_h,\gamma)$. 
	\end{remark}

	The nonlinear term $N(\Theta_h,\gamma)$ is superlinear in $\gamma$ if $p \in (0,\frac{1}{2})$ and quadratic if $p \geq \frac{1}{2}$, according to the following proposition. 
	
	\begin{proposition}
		\label{lemma:nonlin}
		Fix $p > 0$ and $a>0$. If $F:[-a,a]\rightarrow\mathbb{R}$ is defined as 
		$$
		F(x) := (a+x)^{2p+1}-a^{2p+1}-(2p+1)a^{2p}x, 
		$$
		then there exists a positive constant $C>0$, such that for all $x\in[-a,a]$
		\begin{equation}
			|F(x)| \leq \bigg\{\begin{array}{ll}
				C|x|^{2p+1}, & \text{if } p\in(0,\frac{1}{2}),\\
				C|x|^2, & \text{if } p\in [\frac{1}{2},\infty).
			\end{array}
			\label{estimate-on-f}
		\end{equation} 
	\end{proposition}
	
	\begin{proof}
		Without the loss of generality, we assume that $a=1$ 
		due to the scaling transformation:
		$$
		F(x)=a^{2p+1} \left[ \left(1+\frac{x}{a}\right)^{2p+1} -1 -\left(2p+1\right)\frac{x}{a} \right].
		$$
		Note that $F''(x)=2p(2p+1)(1+x)^{2p-1}$, so that if $2p -1 \geq 0$, then $F''$ is bounded on $[-1,1]$, and the second line of (\ref{estimate-on-f}) follows from Taylor's theorem. If $2p - 1 < 0$, then $F''$ is bounded on $[-\frac{1}{2},\frac{1}{2}]$ so  that
		\begin{equation*}
			|F(x)|\leq C|x|^2 \leq C|x|^{2p+1}, \quad x \in \left[-\frac{1}{2},\frac{1}{2} \right],
		\end{equation*}
		for some positive constant $C$. For $|x| \in [\frac{1}{2},1]$, we have $\left(\frac{1}{2}\right)^{2p+1}\leq |x|^{2p+1}$ so that 
		\begin{equation*}
			|F(x)| = \left(\frac{1}{2}\right)^{2p+1} 2^{2p+1} |F(x)| 
			\leq C|x|^{2p+1}, 
			\quad |x| \in \left[\frac{1}{2},1\right], 
		\end{equation*}
		for another positive constant $C$. The two estimates above give the second line of (\ref{estimate-on-f}).
	\end{proof}
	
	The homogeneous equation $\mathcal{L}\gamma=0$ admits two linearly independent solutions. The first one is given by  $\Theta_h'(t)$ due to 
	the translation symmetry of the autonomous equation (\ref{eq:theta_nonlinear}). The other solution denoted by $\Sigma(t)$ can be obtained from the Wronskian relation
	\begin{equation}
		\Theta_h'(t)\Sigma'(t)-\Theta_h''(t)\Sigma(t)=\Sigma_0, \qquad t \in \mathbb{R},
		\label{eq:Wronsk_Theta_Sigma}
	\end{equation}
	where $\Sigma_0\neq 0$ is constant. We take $\Sigma_0=1$ for normalizing  $\Sigma(t)$. Using 
	(\ref{eq:Theta_asympt}) in \eqref{eq:Wronsk_Theta_Sigma}, we obtain that
	\begin{equation}
		\Sigma(t) = -\frac{p^2}{2} \left\{ \begin{array}{ll} 
e^{-\frac{t}{p}} [1+\mathcal{O}(e^{2t})] \quad & \text{ as } \;\;  t\to-\infty, \\
\alpha_p^{\frac{1}{p}}e^{\frac{t}{p}} [1+\mathcal{O}(e^{-2t})] \quad & \text{ as } \;\; t\to+\infty. \end{array} \right.
		\label{eq:Sigma_asympt}
	\end{equation}
	Using the two linearly independent solutions $\Theta_h'$ and $\Sigma$ of the  homogeneous equation $\mathcal{L}\gamma=0$, we rewrite \eqref{eq:L_gam} as an integral equation for $\gamma$:
	\begin{equation}
		\gamma(t) = \int_{-\infty}^t \left(\Theta_h'(t')\Sigma(t)-\Theta_h'(t)\Sigma(t')\right)\left[f_b(t')\left(\Theta_h(t')+\gamma(t')\right)-N\left(\Theta_h(t'),\gamma(t')\right)\right]dt',
		\label{eq:gamm_int_eq}
	\end{equation}
	where the free solution $c_1\Theta_h'(t)+c_2\Sigma(t)$ has been set to zero from the requirement that $\gamma(t)$ decays to zero as $t\to-\infty$ faster  than $\Theta_h'(t)$.
	
	The perturbation term $\gamma$ can be estimated to be small in the $L^{\infty}$ norm on the semi-infinite interval $(-\infty, T+ap \log b]$ with fixed $T > 0$ and $a > 0$, where the right end point diverges asymptotically to $+\infty$ as $b\to\infty$. The following lemma gives the persistence result 
	for the solution $\Psi_b(t-p\log b)$ to stay close to the leading-order 
	term $\Theta_h(t)$ for $t \in (-\infty,T+ap\log b]$.
	
	\begin{lemma}\label{lemma:gamm_est}
		Fix $p \in (0,1]$ and $\lambda\in\mathbb{R}$. For any fixed $T>0$ and $a\in(0,\frac{p}{1+p})$ there exist $b_{T,a}>0$ and $C_{T,a}>0$ such that the unique solution $\Psi_b(t)$ to the second-order equation \eqref{eq:Psi_eq} with asymptotic behaviour \eqref{eq:Psi_b_asympt} satisfies for $t\in(-\infty,T+ap\log b]$ and all $b\geq b_{T,a}$:
		\begin{equation}
|\Psi_b(t-p\log b)-\Theta_h(t)| \leq C_{T,a}b^{-2p(1-a)} e^{\frac{t}{p}},
			\label{eq:gamma_bound}
		\end{equation}
		where the bound can be differentiated in $t$.
	\end{lemma}
	
	\begin{remark}
		The remainder term in the bound (\ref{eq:gamma_bound}) is small 
		on $[ap\log b,T + ap\log b]$ for every $a \in (0,\frac{2p}{1+2p})$ for which $b^{-2p+a(2p+1)}\to 0$ as $b\to\infty$. However, $\Theta_h(t) = \mathcal{O}(b^{-a})$ on the same interval so that the remainder term $\gamma(t)$ is smaller than the leading-order term $\Theta_h(t)$ for $t\in [ap\log b,T + ap\log b]$  if $a\in(0,\frac{p}{1+p})$ for which $b^{-2p+a(2p+1)} \ll b^{-a}$ for sufficiently large $b$.
	\end{remark}

	\begin{proof}
		In order to eliminate the divergence of the integral kernel in the integral equation \eqref{eq:gamm_int_eq} as $t\to - \infty$, we introduce a change of variables: $\tilde{\Theta}_h(t) := e^{-\frac{t}{p}} \Theta_h(t)$ and  $\tilde{\gamma}(t) := e^{-\frac{t}{p}}\gamma(t)$. The new integral equation for $\tilde{\gamma}$ can be considered as the fixed-point equation $\tilde{\gamma} = A \tilde{\gamma}$, where 
		\begin{equation}
			(A\tilde{\gamma})(t) := \int_{-\infty}^t K(t,t')\left[f_b(t')(\tilde{\Theta}_h(t')+\tilde{\gamma}(t'))-
			e^{2t'}N(\tilde{\Theta}_h(t'),\tilde{\gamma}(t'))\right]dt',
			\label{eq:gamm_tild_int_eq}
		\end{equation}
		where the integral kernel is defined as
		\begin{equation*}
			K(t,t') := e^{\frac{t'}{p}}\Theta_h'(t')e^{-\frac{t}{p}}\Sigma(t) - e^{-\frac{t}{p}}\Theta_h'(t)e^{\frac{t'}{p}}\Sigma(t'), \quad t'\leq t.
		\end{equation*}
		It follows from the asymptotic behaviours \eqref{eq:Theta_asympt} and \eqref{eq:Sigma_asympt} for $\Theta_h$ and $\Sigma$ that 
		the integral kernel is bounded for every $t \in \mathbb{R}$ by 
		\begin{equation*}
			|K(t,t')|\leq C \left(1+e^{-\frac{2}{p}(t-t')}\right), \quad t' \leq t,
		\end{equation*}
		for some positive constant $C$. Thus, as the integration in \eqref{eq:gamm_tild_int_eq} is done in $t'$ from $-\infty$ to $t$, the kernel $K(t,t')$ is bounded. In addition, the lower bound on $\tilde{\Theta}_h$ follows from (\ref{eq:Theta_asympt}):
		\begin{equation*}
			\tilde{\Theta}_h(t) \geq C_{T,a} b^{-2a}, \quad t \in (-\infty,T+ap \log b),
		\end{equation*}
		for some positive constant $C_{T,a}$ that depends on $T$ and $a$ for all large $b$. We shall prove that the integral operator $A$ is a contraction in a small closed ball in the Banach space $L^\infty(-\infty,T+ap\log b)$ 
		equipped with the norm $\| \cdot \|_{\infty}$. \\
		
		\textbf{Case $p \in (0,\frac{1}{2})$.} Since $\tilde{\Theta}_h(t)$ is bounded from below for $t \in (-\infty,T+ap \log b)$, it follows by Proposition \ref{lemma:nonlin} if $\| \tilde{\gamma} \|_{\infty} \ll b^{-2a}$ for all large $b$, then
		\begin{equation}
\| N(\tilde{\Theta}_h,\tilde{\gamma}) \|_\infty\leq C\|\tilde{\gamma}\|_\infty^{2p+1},
			\label{eq:gamm_N_bound}	
		\end{equation}
		for some positive constant $C$. We use 
		$f_b(t) := b^{-2p} e^{2t} [ -\lambda + b^{-2p}e^{2t}]$ and estimate
		\begin{align*}
			\|A\tilde{\gamma}\|_\infty &\leq C \left[ \left(1+\|\tilde{\gamma}\|_\infty\right)\int_{-\infty}^{T+ap\log b}|f_b(t')|dt'+\|\tilde{\gamma}\|_\infty^{2p+1}\int_{-\infty}^{T+ap\log b}e^{2t'}dt' \right] \\
			&\leq C\left[ \left(1+\|\tilde{\gamma}\|_\infty\right)b^{-2p(1-a)} + b^{2ap}\|\tilde{\gamma}\|_\infty^{2p+1} \right],
		\end{align*}
		where the positive constant $C$ can change from one line to the other line. If $\|\tilde{\gamma}\|_\infty \leq 2Cb^{-2p(1-a)}$, then
		\begin{align*}
			\|A\tilde{\gamma}\|_\infty & \leq C \left[ b^{-2p(1-a)} + 2Cb^{-4p(1-a)} + (2C)^{2p+1}b^{-2p(1-2a(p+1)+2p)} \right] \\
			& \leq 2C b^{-2p(1-a)},
		\end{align*}
		where we have used $2p(1-a) < 2p(1-2a(p+1)+2p)$ if $a\in(0,\frac{2p}{1+2p})$. 
		Since $2a < 2p (1-a)$ if $a \in (0,\frac{p}{1+p})$ with $\frac{p}{1+p} < \frac{2p}{1+2p}$, the bound $\|\tilde{\gamma}\|_\infty \leq 2Cb^{-2p(1-a)}$ 
		ensures validity of the bound $\| \tilde{\gamma} \|_{\infty} \ll b^{-2a}$ for which the bound (\ref{eq:gamm_N_bound}) can be used.  Similar calculations show that for two functions $\tilde{\gamma}$ and $\tilde{\gamma}'$ in the same small closed ball in $L^{\infty}(-\infty,T+ap\log b)$, we have 
		\begin{equation*}
			\|A\tilde{\gamma}-A\tilde{\gamma}'\|_\infty \leq C b^{-2p(1-a)}\|\tilde{\gamma}-\tilde{\gamma}'\|_\infty,
		\end{equation*}
		so that $A$ is a contraction for sufficiently large $b$. By the Banach fixed-point theorem, there exists a unique fixed point $\tilde{\gamma}$ of $A$ such that
		\begin{equation*}
			\sup_{t\in(-\infty,T+ap\log b)}|\tilde{\gamma}(t)| \leq 2Cb^{-2p(1-a)}.
		\end{equation*}
		Since $\gamma(t)=e^{\frac{t}{p}}\tilde{\gamma}(t)$, we obtain  the bound \eqref{eq:gamma_bound} for the unique solution $\gamma(t)$ to the integral equation \eqref{eq:gamm_int_eq}.\\
		
		\textbf{Case $p \in [\frac{1}{2},1]$.} The only difference in the proof is that, by Proposition \ref{lemma:nonlin}, the bound \eqref{eq:gamm_N_bound} is replaced by the bound 
		\begin{equation}
\| N(\tilde{\Theta}_h,\tilde{\gamma}) \|_\infty \leq C\|\tilde{\gamma}\|_\infty^{2},
		\label{eq:gamm_N_bound-new}	
		\end{equation}
if $\| \tilde{\gamma} \|_{\infty} \ll b^{-2a}$ for all large $b$. In this case, we get the estimate 
		\begin{align*}
			\|A\tilde{\gamma}\|_\infty 
			&\leq C\left[ \left(1+\|\tilde{\gamma}\|_\infty\right)b^{-2p(1-a)} + b^{2ap}\|\tilde{\gamma}\|_\infty^{2} \right], \\
			& \leq C \left[ b^{-2p(1-a)} + 2Cb^{-4p(1-a)} + (2C)^{2} b^{-2p(2-3a)} \right] \\
			& \leq 2C b^{-2p(1-a)},
		\end{align*}
		where we have used $2p(1-a) < 2p(2-3a)$ if $a\in(0,\frac{1}{2})$. 
		Since $2a < 2p (1-a)$ if $a \in (0,\frac{p}{1+p})$ with $\frac{p}{1+p} \leq \frac{1}{2}$, the bound $\|\tilde{\gamma}\|_\infty \leq 2Cb^{-2p(1-a)}$ 
		ensures validity of the bound $\| \tilde{\gamma} \|_{\infty} \ll b^{-2a}$ for which the bound (\ref{eq:gamm_N_bound-new}) can be used. The rest of the proof is verbatim to the case of $p \in (0,\frac{1}{2})$.
	\end{proof}

	\begin{remark}
	The bound (\ref{eq:gamma_bound}) can be extended for every $p \geq 1$ if the values of $a$ are restricted to $a \in (0,\frac{1}{2})$ as follows 
	from the proof of Lemma \ref{lemma:gamm_est} in the case of $p \in [\frac{1}{2},1]$.
\end{remark}
	
	The result of Lemma \ref{lemma:gamm_est} allows us to justify the validity 
	of 
	$$
	\Psi_b(t-p \log b) \sim \Theta_h(t) \sim \alpha_p^{-\frac{1}{p}} e^{-\frac{t}{p}}, \quad t \in [a p \log b, T + a p \log b]
	$$ 
	due to (\ref{eq:Theta_asympt}) and (\ref{eq:gamma_bound}). In order to obtain 
	the correction term which behaves like $e^{\frac{t}{p}}$ in the same asymptotic region, we need to analyze $\gamma$ in more details 
	and obtain the leading-order part of $\gamma$. To do so, we write
	$\gamma = \gamma_h + \delta$, where the leading-order term $\gamma_h$ satisfies $\mathcal{L}\gamma_h = f_b\Theta_h$ and is given explicitly by
	\begin{equation}\label{eq:gamma1_def}
		\gamma_h(t) = \Sigma(t)\int_{-\infty}^tf_b(t')\Theta_h'(t')\Theta_h(t')dt' - \Theta_h'(t)\int_{-\infty}^tf_b(t')\Sigma(t')\Theta_h(t')dt',
	\end{equation}
	whereas the correction term $\delta$ satisfies
	\begin{equation}
			\label{eq:delta_int}
		\mathcal{L} \delta = f_b(\gamma_h+\delta) - N(\Theta_h,\gamma_h+\delta).
	\end{equation}
The following lemma gives a sharper bound on $\gamma_h$ compared to the bound (\ref{eq:gamma_bound}). The sharper bound holds on $[ap \log b,T + ap \log b]$, where the asymptotic behavior of $\Theta_h(t)$ and $\Sigma(t)$ as $t \to +\infty$ is relevant. 

\begin{lemma}
	\label{lem-aux}
	Fix $p \in (0,1)$ and $\lambda \in \mathbb{R}$. For any fixed $T > 0$ and $a \in (0,\frac{p}{1+p})$ there exist $b_{T,a}>0$ and $C_{T,a}>0$ such that 
	$\gamma_h$ in (\ref{eq:gamma1_def}) satisfies for $t\in [ap \log b,T+ap\log b]$ and all $b \geq b_{T,a}$:
	\begin{itemize}
	\item if $p \in (0,\frac{1}{2})$, then
	\begin{equation}
	|\gamma_h(t)| \leq C_{T,a} \left[ (|\lambda| b^{-2p(1-a)} + b^{-4p(1-a)}) e^{-\frac{t}{p}} + (|\lambda| b^{-2p} + b^{-4p}) e^{\frac{t}{p}} \right], 
	\label{eq:gamma_bound-sharper-1}
	\end{equation}
		\item if $p \in [\frac{1}{2},1)$, then
	\begin{equation}
	|\gamma_h(t)| \leq C_{T,a} \left[ (|\lambda| b^{-2p(1-a)} + b^{-4p(1-a)}) e^{-\frac{t}{p}} + (|\lambda| b^{-2p} + b^{-4p(1-a)}) e^{\frac{t}{p}} \right], 
	\label{eq:gamma_bound-sharper-2}
	\end{equation}
	\end{itemize}
	where the bounds can be differentiated in $t$.
\end{lemma}

\begin{proof}
Since $\Sigma(t) \Theta_h(t)$ is bounded for every $t \in \mathbb{R}$ independently of $b$, the second integral term in (\ref{eq:gamma1_def}) is controlled by 
\begin{align*}
\left| \int_{-\infty}^{T+ap\log b}f_b(t)\Sigma(t)\Theta_h(t)dt \right| &\leq C \int_{-\infty}^{T+ap\log b}\left( |\lambda| b^{-2p}e^{2t} + b^{-4p}e^{4t}\right)dt \\
& \leq C_{T,a}  (|\lambda| b^{-2p(1-a)} + b^{-4p(1-a)}).
\end{align*}
This estimate yields the first term in the bounds (\ref{eq:gamma_bound-sharper-1}) and (\ref{eq:gamma_bound-sharper-2}) due to 
the asymptotic behavior (\ref{eq:Theta_asympt}). On the other hand, 
since $\Theta_h(t)^2 = \mathcal{O}(e^{-\frac{2t}{p}})$ as $t \to +\infty$, 
the first integral term in (\ref{eq:gamma1_def}) is controlled by 
\begin{align*}
\left| \int_{-\infty}^{T+ap\log b}f_b(t)\Theta_h'(t)\Theta_h(t)dt \right| &\leq C \int_{-\infty}^{T+ap\log b}\left( |\lambda| b^{-2p}e^{-\frac{2(1-p)}{p}t} + b^{-4p}e^{-\frac{2(1-2p)}{p}t}\right)dt \\
& \leq C_{T,a}  (|\lambda| b^{-2p} + b^{-4p+2a \nu_p}),
\end{align*}	
where $\nu_p = 0$ for $p \in (0,\frac{1}{2})$ and $\nu_p = 2p-1$ for $p \in [\frac{1}{2},1)$. This yields the second term in the bounds (\ref{eq:gamma_bound-sharper-1}) and (\ref{eq:gamma_bound-sharper-2}) due to 
the asymptotic behavior (\ref{eq:Sigma_asympt}), where we have also used that 
$4p(1-a) < 4p - 2a (2p-1)$.
\end{proof}
	
The sharper bounds (\ref{eq:gamma_bound-sharper-1}) and (\ref{eq:gamma_bound-sharper-2}) are compatible with the bound 
(\ref{eq:gamma_bound}) on the semi-infinite interval $(-\infty,T+ap \log b]$, which can be rewritten in the form:
\begin{align}
|\gamma_h(t)| \leq C_{T,a} b^{-2p(1-a)} e^{\frac{t}{p}}, \qquad t \in (-\infty,T+ap \log b].
\label{eq:gamm_1_est}
\end{align}
The correction term $\delta$ is estimated to be smaller than $\gamma_h$ according to the following lemma. 
	
	\begin{lemma}\label{lemma:gamma_est_hot}
		Fix $p \in (0,1]$ and $\lambda\in\mathbb{R}$. For any fixed $T>0$, $a\in(0,\frac{p}{1+p})$, there exist $b_{T,a}>0$ and $C_{T,a}>0$ such that for $t \in (-\infty,T+ap\log b]$ and all $b \geq b_{T,a}$:
		\begin{itemize}
			\item if $p \in (0,\frac{1}{2})$, then
			\begin{equation}\label{eq:gamm2_est_c1}
|\Psi_b(t-p\log b)-\Theta_h(t)-\gamma_h(t)| \leq C_{T,a} b^{-2p\left[ (2p+1)(1-a) - a\right]} e^{\frac{t}{p}}, 
			\end{equation}
			\item if $p \in [\frac{1}{2},1]$, then
			\begin{equation}\label{eq:gamm2_est_c2}
|\Psi_b(t-p\log b)-\Theta_h(t)-\gamma_h(t)| \leq C_{T,a}b^{-2p(2-3a)} e^{\frac{t}{p}},
			\end{equation}
		\end{itemize}
where the bounds can be differentiated in $t$.
	\end{lemma}
	
	\begin{remark}
		\label{rmrk:a_restr}
		If $a \in (0,\frac{2p}{1+2p})$ for $p \in (0,\frac{1}{2})$, we have 
		$$
2p(1-a) < 2p \left[(2p+1)(1-a)-a\right],
		$$
so that comparison (\ref{eq:gamm_1_est}) and \eqref{eq:gamm2_est_c1}  
shows		that $\delta$ is smaller than $\gamma_h$ 
		for sufficiently large $b$. Similarly, if $a \in (0,\frac{1}{2})$ for $p \geq \frac{1}{2}$, we have 
		$$
2p(1-a) < 2p(2-3a),
		$$
		so that the comparison of (\ref{eq:gamm_1_est}) and \eqref{eq:gamm2_est_c2}  shows that $\delta$ is smaller than $\gamma_h$ 
		for sufficiently large $b$. In both cases, by Lemma \ref{lemma:gamm_est},  we also have 
		$\gamma_h$ being smaller than $\Theta_h$ if $a \in (0,\frac{p}{1+p})$, where $\frac{p}{1+p} \leq \min\{ \frac{2p}{1+2p},\frac{1}{2}\}$ if $p \in (0,1]$.
	\end{remark}
	
	\begin{proof}
Equation (\ref{eq:delta_int}) for $\delta$ can be written similarly to (\ref{eq:gamm_int_eq}) as the integral equation
\begin{equation}	
		\delta(t) = \int_{-\infty}^t \left(\Theta_h'(t')\Sigma(t)-\Theta_h'(t)\Sigma(t')\right)\left[f_b(t')\left(\gamma_h(t')+\delta(t')\right)-N\left(\Theta_h(t'),\gamma_h(t')+\delta(t')\right)\right]dt'.
		\label{eq:delta}
\end{equation}
We proceed in a similar way to the proof of Lemma \ref{lemma:gamm_est}. 
		Using the change of variables  
$$
\tilde{\Theta}_h(t) := e^{-\frac{t}{p}} \Theta_h(t), \quad 
\tilde{\gamma}_h(t) := e^{-\frac{t}{p}} \gamma_h(t), \quad 
\tilde{\delta}(t) :=e^{-\frac{t}{p}}\delta(t),
$$ 
		we rewrite the integral equation for $\delta$ as 
		the fixed-point equation $\tilde{\delta} = B \tilde{\delta}$, 
		where 
		\begin{equation}
			(B\tilde{\delta})(t) := \int_{-\infty}^t K(t,t')\left[f_b(t')(\tilde{\gamma}_h(t')+\tilde{\delta}(t')) -e^{2t'} N(\tilde{\Theta}_h(t'),\tilde{\gamma}_h(t')+\tilde{\delta}(t'))\right] dt'.
			\label{eq:gamm_2_tild_int_eq}
		\end{equation}
		The only essential difference between $A$ in \eqref{eq:gamm_tild_int_eq} and $B$ in \eqref{eq:gamm_2_tild_int_eq} is the source term which dictates the size of the closed ball in $L^\infty(-\infty,T+ap\log b]$, where the fixed-point iterations are closed. In \eqref{eq:gamm_2_tild_int_eq}, it consists of the linear term $f_b \tilde{\gamma}_h$ and the contribution from nonlinearity term $N(\tilde{\Theta}_h,\tilde{\gamma}_h)$. The linear term is estimated from (\ref{eq:gamm_1_est}) as 
		\begin{equation}\label{eq:gamm1_tild_fb_est}
			\left| \int_{-\infty}^{T+ap\log b} f_b(t)\tilde{\gamma}_h(t)dt \right| \leq C\|\tilde{\gamma}_h\|_\infty \int_{-\infty}^{T+ap\log b} |f_b(t)| dt \leq C b^{-4p(1-a)}.
		\end{equation}
		Estimates for the nonlinear term depend on the value of $p$. To proceed with the estimates, we decompose 
		\begin{equation*}
			N(\tilde{\Theta}_h,\tilde{\gamma}_h+\tilde{\delta}) = N(\tilde{\Theta}_h,\tilde{\gamma}_h) + \tilde{N}(\tilde{\Theta}_h,\tilde{\gamma}_h,\tilde{\delta}),
		\end{equation*}
		where
		\begin{align*}
			\tilde{N}(\tilde{\Theta}_h,\tilde{\gamma}_h,\tilde{\delta}) &:= N(\tilde{\Theta}_h,\tilde{\gamma}_h+\tilde{\delta}) - N(\tilde{\Theta}_h,\tilde{\gamma}_h) \\ 
			&= (\tilde{\Theta}_h+\tilde{\gamma}_h+\tilde{\delta})^{2p+1} - (\tilde{\Theta}_h +\tilde{\gamma}_h)^{2p+1}-(2p+1)(\tilde{\Theta}_h)^{2p} \tilde{\delta}.
		\end{align*}
		
		\textbf{Case $p \in (0,\frac{1}{2})$.} By Proposition \ref{lemma:nonlin}, we have 
		\begin{equation}\label{eq:N1_c1_est}
			\|N(\tilde{\Theta}_h,\tilde{\gamma}_h)\|_\infty \leq C'\|\tilde{\gamma}_h\|_\infty^{2p+1}.
		\end{equation}
		Since $\tilde{N} |_{\delta = 0} = 0$ and 
		\begin{equation*}
			\frac{\partial \tilde{N}}{\partial \tilde{\delta}}\bigg|_{\tilde{\delta}=0} = (2p+1)(\tilde{\Theta}_h+\tilde{\gamma}_h)^{2p} - (2p+1)(\tilde{\Theta}_h)^{2p},
		\end{equation*}
we obtain by a minor modification of the proof of Proposition \ref{lemma:nonlin}  that 
		\begin{equation}\label{eq:N2_c1_est}
			\|\tilde{N}(\tilde{\Theta}_h,\tilde{\gamma}_h,\tilde{\delta}) \|_\infty \leq C \|\tilde{\gamma}_h\|_\infty^{2p}\|\tilde{\delta}\|_\infty.
		\end{equation}
		Putting together estimates  \eqref{eq:gamm1_tild_fb_est}, \eqref{eq:N1_c1_est}, and \eqref{eq:N2_c1_est}, we obtain that 
		\begin{align*}
			\| B \tilde{\delta} \|_{\infty} \leq C \left(b^{-4p(1-a)} + b^{-2p(1-a)}\|\tilde{\delta}\|_\infty + b^{-2p [(2p+1)(1-a)-a]} + b^{-2p\left[2p(1-a)-a\right]}\|\tilde{\delta}\|_\infty\right).
		\end{align*}
Since $4p(1-a) > 2p [(2p+1)(1-a)-a]$ for every $p \in (0,\frac{1}{2})$, we have $b^{-4p(1-a)} \ll b^{-2p [(2p+1)(1-a)-a]}$ for sufficiently large $b$, hence the source term coming from the nonlinearity $N(\tilde{\Theta}_h,\tilde{\gamma}_h)$ is much larger than the source 
term coming from $f_b \tilde{\gamma}_h$ as $b\to\infty$. 
As a result, if $\| \tilde{\delta}\|_{\infty} \leq 2C b^{-2p [2p(1-a)-a]}$, 
then $\| B \tilde{\delta} \|_{\infty} \leq 2C b^{-2p [2p(1-a)-a]}$. 
Moreover, $B$ is a contraction in the same small closed ball in  $L^\infty(-\infty,T+ap\log b)$ for sufficiently large $b$. Hence, there exists a unique fixed point $\tilde{\delta}$ of $B$ satisfying
$\|\tilde{\delta} \|_{\infty} \leq 2 C b^{-2p [(2p+1)(1-a)-a]}$,
which yields \eqref{eq:gamm2_est_c1} for $\delta(t) = e^{\frac{t}{p}} \tilde{\delta}(t)$.\\

\textbf{Case $p \in [\frac{1}{2},1]$.}  By Proposition \ref{lemma:nonlin}, we have
		\begin{equation*}
			\|N(\tilde{\Theta}_h,\tilde{\gamma}_h)\|_\infty \leq C \|\tilde{\gamma}_h\|_\infty^2,
		\end{equation*}
		and similarly,
		\begin{equation*}
			\| \tilde{N}(\tilde{\Theta}_h,\tilde{\gamma}_h,\tilde{\delta}) \|_\infty \leq C \|\tilde{\gamma}_h\|_\infty\|\tilde{\delta}\|_\infty.
		\end{equation*}
		Proceeding similarly to the previous computations, we obtain
		\begin{align*}
			\| B \tilde{\delta} \|_{\infty} \leq C' \left(b^{-4p(1-a)} + b^{-2p(1-a)}\|\tilde{\delta}\|_\infty + b^{-2p(2-3a)} + b^{-2p(1-2a)}\|\tilde{\delta}\|_\infty\right).
		\end{align*}
Since $4p(1-a) > 2p(2-3a)$, we have $b^{-4p(1-a)} \ll b^{-2p(2-3a)}$ for sufficiently large $b$, hence again 
		the source term coming from the nonlinearity 
		$N(\tilde{\Theta}_h,\tilde{\gamma}_h)$ is much larger than the source 
		term coming from $f_b \tilde{\gamma}_h$ as $b\to\infty$. 
		Proceeding similarly, for sufficiently large $b$, there exists a unique fixed point $\tilde{\delta}$ of $B$ satifying 
$\|\tilde{\delta}\|_{\infty} \leq 2 C b^{-2p(2-3a)}$,
which  yields \eqref{eq:gamm2_est_c2} for $\delta(t) = e^{\frac{t}{p}} \tilde{\delta}(t)$.
	\end{proof}

Similarly to Lemma \ref{lem-aux}, we can find a sharper bound on $\delta$ compared to the bounds (\ref{eq:gamm2_est_c1}) and (\ref{eq:gamm2_est_c2}). 
This is given by the following lemma, the proof of which follows from 
the estimates obtained in Lemma \ref{lemma:gamma_est_hot}.

\begin{lemma}
	\label{lem-aux-new}
	Fix $p \in (0,1)$ and $\lambda \in \mathbb{R}$. For any fixed $T > 0$ and $a \in (0,\frac{p}{1+p})$ there exist $b_{T,a}>0$ and $C_{T,a}>0$ such that 
	$\delta$ in (\ref{eq:delta}) satisfies for $t\in [ap \log b,T+ap\log b]$ and all $b \geq b_{T,a}$:
\begin{itemize}
	\item if $p \in (0,\frac{1}{2})$, then
	\begin{equation}
	|\delta(t)| \leq C_{T,a} (|\lambda| b^{-2p} + b^{-4p})^{2p+1} b^{2ap} e^{\frac{t}{p}}, 
	\label{eq:delta_bound-sharper-1}
	\end{equation}
	\item if $p \in [\frac{1}{2},1)$, then
	\begin{equation}
	|\gamma_h(t)| \leq C_{T,a} (|\lambda| b^{-2p} + b^{-4p(1-a)})^{2} b^{2ap} e^{\frac{t}{p}},
	\label{eq:delta_bound-sharper-2}
	\end{equation}
\end{itemize}
where the bounds can be differentiated in $t$.
\end{lemma}

\begin{proof}
	For $p \in (0,\frac{1}{2})$, the first term in the bound (\ref{eq:gamma_bound-sharper-1}) with $e^{-\frac{t}{p}}$ is much smaller than the second term with $e^{\frac{t}{p}}$ on $[ap \log b, T + ap \log b]$. As a result, it can be neglected. On the other hand, the second term in the bound (\ref{eq:gamma_bound-sharper-1}) can be extended for the semi-infinite interval $(-\infty,T + ap \log b]$ such that 
	the sharper bound compared to (\ref{eq:gamma_bound}) can be written in the form 
		\begin{equation*}
|\gamma_h(t)| \leq C_{T,a} (|\lambda| b^{-2p} + b^{-4p}) e^{\frac{t}{p}}, \quad t \in (-\infty,T + a p \log b].
\end{equation*}	
The bound (\ref{eq:delta_bound-sharper-1}) follows from 
analysis of the integral equation (\ref{eq:delta}) by using the transformation to the tilde variables 
in the proof of Lemma \ref{lemma:gamma_est_hot} and the estimate (\ref{eq:N1_c1_est}) on the nonlinear term $N$ which is much larger than the source term from $f_b$.

For $p \in [\frac{1}{2},1)$, the proof is analogous but we use 
		\begin{equation*}
|\gamma_h(t)| \leq C_{T,a} (|\lambda| b^{-2p} + b^{-4p(1-a)}) e^{\frac{t}{p}}, \quad t \in (-\infty,T + a p \log b].
\end{equation*}	
and the estimate (\ref{eq:N2_c1_est}) on the nonlinear term $N$ which is still much larger than the source term from $f_b$.
\end{proof}
	
	\section{Persistence of the $c$-family of solutions}
	\label{sec-4}
	
	The $c$-family of solutions $\Psi_c$ of the differential equation (\ref{eq:Psi_eq}) satisfying \eqref{far-field} is considered near the solution $\Upsilon_h$ of the linear equation (\ref{eq:Upsilon_eq}) given by (\ref{eq:Ups_h}). The comparison gives $\Psi_c(t) \sim c \Upsilon_h(t)$ as $t \to +\infty$. The correction term $\eta(t):=\Psi_c(t)-c \Upsilon_h(t)$ satisfies
	\begin{equation}
		\label{eq:eta}
		\mathcal{M} \eta =  -\left| c \Upsilon_h +\eta \right|^{2p} (c \Upsilon_h +\eta),
	\end{equation}
	where
	\[
	\left(\mathcal{M} \eta\right)(t) := \eta''(t)-\frac{1}{p^2}\eta(t)+\lambda e^{2t}\eta(t) -e^{4t}\eta(t).
	\]
	The homogeneous equation $\mathcal{M} \eta = 0$ has two linearly independent solutions. One solution is $\Upsilon_h$ given by \eqref{eq:Ups_h}. The other solution, denoted as $\Upsilon_g$, can be obtained from the normalized Wronskian relation
	\begin{equation}
		\label{eq:Ups_wronsk}
		\Upsilon_h(t)\Upsilon_g'(t)-\Upsilon_h'(t)\Upsilon_g(t) = 1.
	\end{equation}
	Since it follows from (\ref{far-field}) that 
	\begin{equation}
		\label{eq:Ups_h_lrg_t}
		\Upsilon_h(t) \sim e^{-\left(1-\frac{\lambda}{2}\right)t}e^{-\frac{1}{2}e^{2t}}, \quad \text{ as }t\to+\infty,
	\end{equation}
	integrating the Wronskian relation \eqref{eq:Ups_wronsk} yields
	\begin{equation}
		\label{eq:Ups_g_lrg_t}
		\Upsilon_g(t) \sim \frac{1}{2}e^{-\left(1+\frac{\lambda}{2}\right)t}e^{\frac{1}{2}e^{2t}}, \quad \text{ as }t\to+\infty.
	\end{equation}
With two linearly independent solutions $\Upsilon_h$ and $\Upsilon_g$, we rewrite \eqref{eq:eta} as an integral equation for $\eta$:
	\begin{equation}
		\label{eq:eta_int}
		\eta(t) = \int_t^{\infty}\left(\Upsilon_h(t')\Upsilon_g(t)-\Upsilon_h(t)\Upsilon_g(t')\right) |c \Upsilon_h(t')+\eta(t')|^{2p} (c \Upsilon_h(t')+\eta(t') ) dt',
	\end{equation}
	where the free solution $c_1\Upsilon_h+c_2\Upsilon_g$ has been set to zero in order to guarantee that $\eta(t)$ decays to zero as $t \to +\infty$ faster than $\Upsilon_h(t)$. 
	
	The following lemma describes the size of $\eta(t)$ for $t \in [0,\infty)$.
	
	\begin{lemma}
		\label{lemma:c_sol_est}
		Fix $\lambda\in(-\infty,2]$ and $p>0$. Then, there exist some constants $C >0$ and $c_0>0$, such that for $t \in [0,\infty)$ and 
		$c \in (-c_0,c_0)$, we have
		\begin{equation}
			\label{eq:Psi_c_est}
|\Psi_c(t)-c \Upsilon_h(t)|\leq C |c|^{2p+1} e^{-(1-\frac{\lambda}{2}) t} e^{-\frac{1}{2} e^{2t}}, 
		\end{equation}
		where the bound can be differentiated term by term. 
	\end{lemma}
	
	\begin{proof}
		In order to obtain a bounded kernel in the integral equation \eqref{eq:eta_int}, we first introduce the change of variables
		\[
		\tilde{\Upsilon}_h(t):= e^{(1-\frac{\lambda}{2})t}e^{\frac{1}{2}e^{2t}} \Upsilon_h(t), \qquad 		\tilde{\eta}(t):= e^{(1-\frac{\lambda}{2})t}e^{\frac{1}{2}e^{2t}}\eta(t),
		\]
		which applied to \eqref{eq:eta_int} results in the integral equation $\tilde{\eta}=E\tilde{\eta}$, where
		\begin{equation}
			\label{eq:eta_tild_int}
			(E\tilde{\eta})(t):=\int_t^{\infty}\hat{K}(t,t') e^{-p(2-\lambda)t' - p e^{2t'}-2t'}\hat{N}(		c\tilde{\Upsilon}_h(t'),\tilde{\eta}(t')) dt',
		\end{equation}
		and where the kernel $\hat{K}$ and the nonlinearity $\hat{N}$ are given by
		\[
		\hat{K}(t,t') := e^{-(1-\frac{\lambda}{2})(t'-t)+2t'+\frac{1}{2}(e^{2t}-e^{2t'})}(\Upsilon_h(t')\Upsilon_g(t)-\Upsilon_h(t)\Upsilon_g(t')),
		\]
		and
		\[
		\hat{N}(c \tilde{\Upsilon}_h,\tilde{\eta}) := |c \tilde{\Upsilon}_h +\tilde{\eta}|^{2p} (c \tilde{\Upsilon}_h+\tilde{\eta}).
		\]
		Using asymptotic behaviours \eqref{eq:Ups_h_lrg_t} and \eqref{eq:Ups_g_lrg_t}, we get that there exists $C > 0$ 
		such that 
\begin{equation}
\label{K-hat}
		|\hat{K}(t,t')| \leq C (1 + e^{\phi(t,t')}), \qquad 0 \leq t \leq t'.
\end{equation}
where $\phi(t,t') := \lambda(t'-t) -e^{2t'} (1-e^{-2(t'-t)})$. Since 
$\phi(t,t') \to -\infty$ as $t' \to +\infty$ and $\phi(t,t')$  has an extremum in $t'$ at $\lambda - 2 e^{2t'} = 0$, which does not belong to $\mathbb{R}$ if $\lambda \in (-\infty,0]$ and is located on $\mathbb{R}_-$ if $\lambda \in (0,2)$, we conclude that $\max\limits_{t' \in [t,\infty)} e^{\phi(t,t')} = e^{\phi(t,t)} = 1$ for $t \geq 0$. Hence, the kernel $\hat{K}(t,t')$ is bounded for every $0 \leq t \leq t' < \infty$. On the other hand, since $\hat{N}(c\Upsilon_h,\tilde{\eta})$ is a $C^1$ function for every $p>0$, the nonlinear term satisfies the following bound:
		\begin{equation}
			\label{eq:N_hat}
|\hat{N}(c	\tilde{\Upsilon}_h,\tilde{\eta})| \leq C |c|^{2p} (|c|  + |\tilde{\eta}|), \quad \mbox{\rm as long as } \; |\tilde{\eta}| \leq C,
		\end{equation}
where the constant $C > 0$ is independent of $c$. 
		
In order to use the Banach fixed-point theorem, we first estimate the size of $E\tilde{\eta}$ for $\tilde{\eta}$ in a small closed ball in $L^\infty(0,\infty)$. Since $e^{-p(2-\lambda)t - p e^{2t}-2t}$ in \eqref{eq:eta_tild_int} is absolutely integrable on $[0,\infty)$, 
we obtain by using \eqref{K-hat} and \eqref{eq:N_hat} that 
		\[
		\|E\tilde{\eta}\|_\infty \leq C |c|^{2p} (|c| + \|\tilde{\eta}\|_\infty),
		\]
		where the constant $C > 0$ is independent of $c$. Thus, the operator $E$ maps a closed ball of radius $2 C |c|^{2p+1}$ into itself  as long as $|c|$ is chosen sufficiently small.
		
		Similarly, by Proposition \ref{lemma:nonlin}, we get 
		for every $\tilde{\eta}_1$ and $\tilde{\eta}_2$ in the same small closed ball 
		in $L^{\infty}(0,\infty)$ that $\hat{N}$ is a Lipschitz function satisfying
\begin{equation}
\label{hat-N-Lipschitz}
		\|\hat{N}(c \tilde{\Upsilon}_h,\tilde{\eta}_1) -\hat{N}(c \tilde{\Upsilon}_h,\tilde{\eta}_2)\|_\infty \leq C |c|^{2p} \|\tilde{\eta}_1-\tilde{\eta}_2\|_\infty,
\end{equation}
		which yields
		\[
		\|E\tilde{\eta}_1-E\tilde{\eta}_2\|_\infty \leq C |c|^{2p} \|\tilde{\eta}_1-\tilde{\eta}_2\|_\infty,
		\]
		so that the operator $E$ is a contraction for sufficiently small values of $|c|$. By the Banach fixed-point theorem, there exists a unique solution $\tilde{\eta}(t)\in L^\infty(0,\infty)$ of the integral equation $\tilde{\eta}=E\tilde{\eta}$ satisfying $\| \tilde{\eta} \|_{\infty} \leq 2 C |c|^{2p+1}$. This estimate yields \eqref{eq:Psi_c_est} after unfolding the transformation for $\eta(t)$.
	\end{proof}

	Using the result of Lemma \ref{lemma:c_sol_est}, we can now extend the estimates for $\eta(t)$ for large negative values of $t$.
	
	\begin{lemma}
		Fix $\lambda\in(-\infty,2]$, $p>0$, and $a\in(0,1)$. Then, there exist $b_0>0$ and $C>0$, such that for every $b\geq b_0$ there exists $c_0>0$, such that for $t \in [-(1-a)p \log b,0]$ and 
$c \in (-c_0 b^{-(1-a)},c_0b^{-(1-a)})$, we have
		\begin{equation}
			\label{eq:Psi_c_est_neg_t}
|\Psi_c(t)-c\Upsilon_h(t)| \leq C |c|^{2p+1} b^{2p(1-a)} e^{-\frac{t}{p}},
		\end{equation}
where the bound can be differentiated term by term. 
		\label{lemma:c_est_long}
	\end{lemma}

\begin{proof}
		We rewrite equation (\ref{eq:eta}) as an integral equation for $t\in[-(1-a)p\log b,0]$:
		\begin{multline}
			\label{eq:eta_int2}
			\eta(t) = \left(\Upsilon_g'(0)\Upsilon_h(t)-\Upsilon_h'(0)\Upsilon_g(t)\right)\eta(0) + \left(\Upsilon_h(0)\Upsilon_g(t)-\Upsilon_g(0)\Upsilon_h(t)\right)\eta'(0) \\ +  \int_t^0\left(\Upsilon_h(t')\Upsilon_g(t)-\Upsilon_h(t)\Upsilon_g(t')\right)\left(c\Upsilon_h(t')+\eta(t')\right)\left|c\Upsilon_h(t')+\eta(t')\right|^{2p}dt',
		\end{multline}
where $|\eta(0)| + |\eta'(0)| \leq C |c|^{2p+1}$ by Lemma \ref{lemma:c_sol_est}.
By using the scattering relation \eqref{scatering-1_simple}  and the transformation \eqref{eq:Ups_h}, we obtain the following asymptotic behavior for $\Upsilon_h(t)$:
		\begin{equation}
		\label{eq:Ups_h_neg_t}
		\Upsilon_h(t) \sim  \frac{\Gamma\left(\frac{1}{p}\right)}{\Gamma(\alpha)}e^{-\frac{t}{p}}, \quad \text{ as } t\to-\infty,
		\end{equation}
where $\alpha>0$ by \eqref{eq:alph}. Wronskian relation \eqref{eq:Ups_wronsk} implies the following asymptotic behavior for  $\Upsilon_g(t)$:
\begin{equation}
\label{eq:Ups_g_neg_t}
		\Upsilon_g(t)\sim\frac{p\Gamma(\alpha)}{2\Gamma\left(\frac{1}{p}\right)}e^{\frac{t}{p}}, \quad \text{ as } t\to-\infty.
		\end{equation}
The divergent behaviour of $\Upsilon_h(t)$ as $t \to -\infty$ dictates the correct form of the transformation to use, which in this case is given by:
		\[
\tilde{\Upsilon}_h(t):=e^{\frac{t}{p}}\Upsilon_h(t), \quad 	
\tilde{\Upsilon}_g(t):=e^{\frac{t}{p}}\Upsilon_g(t), \quad 		\tilde{\eta}(t):=e^{\frac{t}{p}}\eta(t).
		\]
Applying it to the integral equation \eqref{eq:eta_int2} results in the fixed point equation $\tilde{\eta}=\mathcal{F}\tilde{\eta}$, where
\begin{multline}
\label{eq:eta_tild_int2}
(\mathcal{F}\tilde{\eta})(t) := \left(\Upsilon_g'(0)\tilde{\Upsilon}_h(t) - \Upsilon_h'(0)\tilde{\Upsilon}_g(t)\right)\eta(0) + \left(\tilde{\Upsilon}_h(0)\Upsilon_g(t)-\Upsilon_g(0)\tilde{\Upsilon}_h(t)\right)\eta'(0) \\ +  \int_t^0 e^{-2t'} \tilde{K}(t,t') \hat{N}(c\tilde{\Upsilon}_h(t'),\tilde{\eta}(t'))dt',
\end{multline}
where 
$$
\tilde{K}(t,t') := e^{\frac{t-t'}{p}} \left(\Upsilon_h(t')\Upsilon_g(t)-\Upsilon_h(t)\Upsilon_g(t')\right)
$$
and $\hat{N}$ is the same as in the proof of Lemma \ref{lemma:c_sol_est}.
		
We proceed by estimating each term of $\mathcal{F}\tilde{\eta}$ in the space $L^\infty(-(1-a)p\log b,0)$. Since $\tilde{\Upsilon}_h$ and $\tilde{\Upsilon}_g$ are bounded for $t\in(-\infty,0]$ and 
		$|\eta(0)| + |\eta'(0)| \leq C |c|^{2p+1}$, we obtain
		\begin{equation*}
		\left|\left(\Upsilon_g'(0)\tilde{\Upsilon}_h(t)-\Upsilon_h'(0)\tilde{\Upsilon}_g(t)\right)\eta(0) + \left(\tilde{\Upsilon}_h(0)\Upsilon_g(t)-\Upsilon_g(0)\tilde{\Upsilon}_h(t)\right)\eta'(0)\right| \leq C|c|^{2p+1}.
		\end{equation*}
Furthermore, if $t \ll -1$, asymptotics \eqref{eq:Ups_h_neg_t} and \eqref{eq:Ups_g_neg_t} allow us to estimate size of the last term in \eqref{eq:eta_tild_int2} as
		\begin{align*}
		\left| \int_t^0 e^{-2t'} \tilde{K}(t,t') \hat{N}(		c\tilde{\Upsilon}_h(t'),\tilde{\eta}(t'))dt' \right| &\leq C |c|^{2p} \int_{t}^0 e^{-2t'} (1 + e^{-\frac{2}{p}(t'-t)}) (|c| + |\tilde{\eta}(t')|) dt' \\
		&\leq C |c|^{2p} b^{2p(1-a)} (|c| + \| \tilde{\eta} \|_\infty),
		\end{align*}
where we have used the $C^1$ property of $\hat{N}(c\tilde{\Upsilon}_h,\tilde{\eta})$ satisfying (\ref{eq:N_hat}).
These two estimates yield
		\[
		\|\mathcal{F}\tilde{\eta}\|_\infty \leq C |c|^{2p}b^{2p(1-a)}\left(|c|+\|\tilde{\eta}\|_\infty\right),
		\]
		for sufficiently large values of $b$. The divergent behaviour of $b^{2p(1-a)}$ for large $b$ is controlled by appropriately reducing the value of $|c|$ satisfying $|c| < c_0 b^{-(1-a)}$ for sufficiently small $c_0 > 0$. Thus, we see that the operator $\mathcal{F}$ maps the closed ball of the radius $2C|c|^{2p+1}b^{2p(1-a)}$ in $L^\infty(-(1-a)p\log b,0)$ into itself. Moreover, since $\hat{N}$ is a Lipschitz function satisfying (\ref{hat-N-Lipschitz}), we get that if $\tilde{\eta}_1,\tilde{\eta}_2$ belong to the same ball, then
		\[
			\|\mathcal{F}\eta_1-\mathcal{F}\eta_2\|_\infty \leq C|c|^{2p}b^{2p(1-a)}\|\tilde{\eta}_1-\tilde{\eta}_2\|_\infty,
		\]
		so that the operator $\mathcal{F}$ is a contraction as long as $|c|< c_0 b^{-(1-a)}$ for sufficiently small $c_0 > 0$. By the Banach fixed-point theorem, there exists a unique $\tilde{\eta}\in L^\infty(-(1-a)p\log b,0)$ satisfying
		\[
		\sup_{t\in[-(1-a)p\log b,0]}|\tilde{\eta}(t)|\leq C|c|^{2p+1}b^{2p(1-a)}, 
		\]
which yields the bound \eqref{eq:Psi_c_est_neg_t} due to the transformation $\eta(t)=e^{-\frac{t}{p}}\tilde{\eta}(t)$.
	\end{proof}
	
	\section{Matching solutions: the proof of Theorem \ref{theorem-main}}
	\label{sec-5}

We are now equipped with all the necessary estimates to prove Theorem \ref{theorem-main}. The ground state $u = u_b$ of the stationary Gross--Pitaevskii equation (\ref{GP}) in the Emden-Fowler variables (\ref{eq:EF_transf}) exhibits decaying behaviour both as $t\to-\infty$ and as $t\to+\infty$ for every $b > 0$ if $\lambda = \lambda(b)$. In other words, it appears at the intersection of the two solution families with 
\begin{equation}
\label{intersection}
\Psi_b(t)=\Psi_{c(b)}(t), \qquad t\in\mathbb{R}
\end{equation}
for some $\lambda = \lambda(b)$ and $c = c(b)$. This allows us to use the asymptotic behaviours \eqref{eq:gamm2_est_c1} and \eqref{eq:gamm2_est_c2} for $\Psi_b(t)$, and the asymptotic behavior \eqref{eq:Psi_c_est_neg_t} for $\Psi_c(t)$ at the times $t = T-(1-a)p\log b$ with varying $T > 0$ and sufficiently large values of $b$. Equaling the asymptotic behaviors due to (\ref{intersection}) yields two implicit equations for parameters $\lambda$ and $c$.

Bound (\ref{asymptotics-u-b}) follows from the bound (\ref{eq:gamma_bound}) with $T = 0$ after the transformation (\ref{eq:EF_transf}). Bounds 
(\ref{asymptotics-u-c}) and (\ref{asymptotics-u-d}) follow 
from the bounds (\ref{eq:Psi_c_est}) and (\ref{eq:Psi_c_est_neg_t}) 
in the reversed order after the transformation (\ref{eq:EF_transf}). The proof of Theorem \ref{theorem-main} is completed after obtaining the asymptotic representation for $\lambda(b)$ and $c(b)$ for large $b$.
	
We fix $p\in(0,1)$, $T>0$, and $a\in(0,\frac{p}{1+p})$. For sufficiently large 
$b \geq b_{T,a}$, we consider $(\lambda,c)$ in the rectangle
$[0,2] \times [0,c_0 b^{-(1-a)}]$ for which both Lemmas \ref{lem-aux-new} and \ref{lemma:c_est_long} can be applied.

By Lemma \ref{lem-aux-new}, we have for $t\in [ap \log b,T+a p\log b]$,
	\begin{equation}\label{eq:Psi_b_exp}
		\Psi_b(t-p \log b) = \Theta_h(t) + \gamma_h(t) + \mathcal{O}(F(\lambda,b) b^{2ap} e^{\frac{t}{p}}), 
	\end{equation}
where 
\begin{equation*}
F(\lambda,b) := \left\{ \begin{array}{ll} (|\lambda| b^{-2p} + b^{-4p})^{2p+1}, & \quad  p\in(0,\frac{1}{2}), \\
(|\lambda| b^{-2p} + b^{-4p(1-a)})^2, & \quad p\in[\frac{1}{2},1),
\end{array} \right.
\end{equation*}
and the asymptotic expansion can be differentiated in $t$.
Using \eqref{eq:gamma1_def}, we can write \eqref{eq:Psi_b_exp}  as
\begin{align*}
		\Psi_b(t-p\log b) =& \Theta_h(t) + \Sigma(t)\int_{-\infty}^{t} f_b(t') \Theta_h'(t')\Theta(t')dt'\\ 
		& - \Theta_h'(t)\int_{-\infty}^{t} f_b(t')\Sigma(t')\Theta_h(t')dt' + \mathcal{O}(F(\lambda,b) b^{2ap} e^{\frac{t}{p}}).
\end{align*}
Evaluating these expressions at $t=T+ ap\log b$ and using the asymptotic relations \eqref{eq:Theta_asympt} and \eqref{eq:Sigma_asympt} for $\Theta_h(t)$ and $\Sigma(t)$ as $t \to +\infty$, we obtain
\begin{align*}
		\Psi_b(T-(1-a)p\log b) =& \alpha_p^{-1/p} e^{-\frac{T}{p}} b^{-a} [ 1+\mathcal{O}(b^{-2ap})] \\
		& -\frac{1}{2} p^2\alpha_p^{1/p}e^{\frac{T}{p}} b^a [1+\mathcal{O}(b^{-2ap})] \int_{-\infty}^{T+ap\log b} f_b(t) \Theta_h'(t) \Theta_h(t) dt \\ 
		&+ \frac{1}{p} \alpha_p^{-1/p} e^{-\frac{T}{p}} b^{-a} [1+\mathcal{O}(b^{-2ap})] \int_{-\infty}^{T+ap\log b} f_b(t) \Sigma(t) \Theta_h(t) dt \\
		& + \mathcal{O}(F(\lambda,b) b^{2ap+a} e^{\frac{T}{p}}),
\end{align*}
where $f_b(t) =  -\lambda b^{-2p} e^{2t} + b^{-4p}  e^{4t}$. 
Since 
\begin{align*}
\left| \int_{-\infty}^{T+ap\log b}f_b(t)\Sigma(t)\Theta_h(t)dt \right| \leq C_{T,a}  b^{-2p(1-a)}
\end{align*}
is obtained in the proof of Lemma \ref{lem-aux}, we finally obtain 
the asymptotic formula:
	\begin{align}
		& \Psi_b(T-(1-a)p\log b) = \alpha_p^{-1/p} e^{-\frac{T}{p}} b^{-a} [1 + \mathcal{O}(b^{-2ap}, b^{-2p(1-a)})] \notag \\ & 
		\qquad-\frac{1}{2} p^2\alpha_p^{1/p}e^{\frac{T}{p}} b^a \left[ 
		\int_{-\infty}^{T+ap\log b} f_b(t) \Theta_h'(t) \Theta_h(t) dt [1 + \mathcal{O}(b^{-2ap})] + \mathcal{O}(F(\lambda,b) b^{2ap}) \right].
		\label{eq:Psi_b_match}
	\end{align}

By Lemma \ref{lemma:c_est_long}, we have  for $t\in(-(1-a) p\log b,0]$,	
\begin{equation}
\label{eq:Psi_c_est_prel} 
\Psi_c(t) = c \Upsilon_h(t) + \mathcal{O}(|c|^{2p+1} b^{2p (1-a)} e^{-\frac{t}{p}}),
\end{equation}	
where $\Upsilon_h$ is given by (\ref{eq:Ups_h}) and the asymptotic expansion can be differentiated in $t$. Since the expansion (\ref{eq:Psi_c_est_prel}) 
is used for $t \to -\infty$, we can use either (\ref{scatering-1_simple}) 
or (\ref{scatering-2}) for asymptotic expansions of the Tricomi 
function $\mathfrak{U}(e^{2t};\alpha,\beta)$ in (\ref{eq:Ups_h}), where $\alpha = \frac{p+1}{2p} - \frac{\lambda}{4} > 0$ due to (\ref{eq:alph}) and $\beta = 1 + \frac{1}{p}$. 
If $p \neq \frac{1}{n}$ with $n \in \mathbb{N}$, then 
the asymptotic formula for the solution $\Psi_c$ evaluated at $t =T-(1-a)p\log b$ is obtained with the help of \eqref{scatering-1_simple}, \eqref{eq:Ups_h}, and \eqref{eq:Psi_c_est_prel} in the form:
\begin{align}
\Psi_c(T-(1-a)p\log b) =& c e^{\frac{T}{p}} \frac{\Gamma\left(-\frac{1}{p}\right)}{\Gamma\left(\frac{p-1}{2p}-\frac{\lambda}{4}\right)}b^{-(1-a)} [1+\mathcal{O}(b^{-2p(1-a)})] \notag \\ &+ c e^{-\frac{T}{p}}\frac{\Gamma\left(\frac{1}{p}\right)}{\Gamma\left(\frac{p+1}{2p}-\frac{\lambda}{4}\right)}b^{1-a} [ 1+\mathcal{O}(b^{-2p(1-a)},|c|^{2p} b^{2p(1-a)})]. \label{eq:Psi_c_match}
\end{align}
If $p = \frac{1}{n}$ with $n \in \mathbb{N}$, then the asymptotic formula for the solution $\Psi_c$ evaluated at $t =T-(1-a)p\log b$ is obtained with the help of \eqref{scatering-2}, \eqref{eq:Ups_h}, and \eqref{eq:Psi_c_est_prel} in the form:
\begin{align}
\Psi_c(T-(1-a)p\log b) =& c e^{nT}  \frac{2 (-1)^{n+1}}{n!\Gamma\left(\frac{1-n}{2}-\frac{\lambda}{4}\right)} 
b^{-(1-a)} [ (T - (1-a) p \log b) [1 + \mathcal{O}(b^{-2p(1-a)})] + \mathcal{O}(1) ] \notag \\ 
& + c e^{-nT} \frac{(n-1)!}{\Gamma(\alpha)} b^{1-a} [1 + 
\mathcal{O}(b^{-2p(1-a)},|c|^{2p} b^{2p(1-a)})], \label{eq:Psi_c_t_c1}
\end{align}
where $\alpha = \frac{1+n}{2} - \frac{\lambda}{4} > 0$.\\

When we use the connection equation (\ref{intersection}), it sets up the system of two equations for two unknowns $\lambda$ and $c$. These two equations can be obtained by equaling $\Psi_b(t)$ and $\Psi_c(t)$ as well as their first derivatives at the time $t = T - (1-a)p \log b$. Alternatively, since the asymptotic expansions are differentiable in $t$ term by term, we can set up the system by equaling coefficients in front of the exponential functions 
$e^{\frac{T}{p}}$ and $e^{-\frac{T}{p}}$. Equaling the coefficients for 
the $e^{-\frac{T}{p}}$ terms in (\ref{eq:Psi_b_match}) with either 
(\ref{eq:Psi_c_match}) or (\ref{eq:Psi_c_t_c1}) yields the following equation:
\begin{equation}
\label{eq-1}
\alpha_p^{-1/p} b^{-a} [1 + \mathcal{O}(b^{-2ap}, b^{-2p(1-a)})] =  c \frac{\Gamma\left(\frac{1}{p}\right)}{\Gamma(\alpha)}  b^{1-a} [ 1+\mathcal{O}(b^{-2p(1-a)},|c|^{2p} b^{2p(1-a)})].
\end{equation}
The nonlinear equation (\ref{eq-1}) is defined for $(\lambda,c) \in [0,2] \times [0,c_0 b^{-(1-a)}]$ and the remainder terms are $C^1$ functions 
with respect to $(\lambda,c)$. Since the leading-order part of the nonlinear equation (\ref{eq-1}) is linear in $c$ and suggests 
the solution $c = \mathcal{O}(b^{-1})$, which clearly exists inside 
$|c| \leq c_0 b^{-(1-a)}$, we have by an application of the implicit function 
theorem the existence of a $C^1$ function $c = c(\lambda,b)$ for 
$\lambda \in [0,2]$ and sufficiently large $b \geq b_{T,a}$, which is given asymptotically as 
\begin{equation}
\label{c-solution}
c(\lambda,b)  = \alpha_p^{-1/p} \frac{\Gamma(\alpha)}{\Gamma\left(\frac{1}{p}\right)} b^{-1} [ 1 + \mathcal{O}(b^{-2ap},b^{-2p(1-a)}) ].
\end{equation}
Equaling the coefficients for 
the $e^{\frac{T}{p}}$ terms in (\ref{eq:Psi_b_match}) with either 
(\ref{eq:Psi_c_match}) or (\ref{eq:Psi_c_t_c1}) and substituting 
the expression (\ref{c-solution}) for $c$ yields a nonlinear equation 
for $\lambda$, which we can also solve with an application 
of the implicit function theorem. However, details of computations
depend on the value of $p \in (0,1)$ and hence are reported  
separately for different values of $p$. \\

{\bf Case $p \in (0,\frac{1}{2})$}. If $p \neq \frac{1}{n}$ for every $n \in \mathbb{N}$, we use (\ref{eq:Psi_b_match}) and 
(\ref{eq:Psi_c_match}) in (\ref{intersection}), equal the coefficients for the 
$e^{\frac{T}{p}}$ terms, and substitute
the expression (\ref{c-solution}) for $c = c(\lambda,b)$. This yields the nonlinear equation for $\lambda$:
\begin{align}
& \frac{1}{2} p^2 \alpha_p^{1/p} b^a \big[ 
\lambda b^{-2p} \int_{-\infty}^{T+ap\log b} e^{2t} \Theta_h'(t) \Theta_h(t) dt [1 + \mathcal{O}(b^{-2ap})] \notag \\
& -   b^{-4p} \int_{-\infty}^{T+ap\log b} e^{4t} \Theta_h'(t) \Theta_h(t) dt [1 + \mathcal{O}(b^{-2ap})] + 
\mathcal{O}((|\lambda| b^{-2p} + b^{-4p})^{2p+1} b^{2ap})  \big] \notag \\
 &= 	 \alpha_p^{-1/p} \frac{\Gamma\left(\frac{p+1}{2p}-\frac{\lambda}{4}\right)\Gamma\left(-\frac{1}{p}\right)}{\Gamma\left(\frac{p-1}{2p}-\frac{\lambda}{4}\right)\Gamma\left(\frac{1}{p}\right)} b^{-2+a} [ 1 + \mathcal{O}(b^{-2p(1-a)})]. \label{eq-2}
\end{align}
If $p \in (0,\frac{1}{2})$, both integrals in the left-hand-side of \eqref{eq-2} converge due to the asymptotic expansion 
(\ref{eq:Theta_asympt}) so that they can be expanded as 		
\begin{align*}
\int_{-\infty}^{T+ap\log b} e^{2t} \Theta_h'(t) \Theta_h(t) dt &= -\int_{-\infty}^{+\infty} e^{2t} \Theta_h(t)^2 dt + \mathcal{O}(b^{-2a(1-p)}),\\
\int_{-\infty}^{T+ap\log b} e^{4t} \Theta_h'(t) \Theta_h(t) dt &= -2\int_{-\infty}^{+\infty} e^{4t} \Theta_h(t)^2 dt + \mathcal{O}(b^{-2a(1-2p)}),
\end{align*}
which implies that the nonlinear equation (\ref{eq-2}) for $\lambda$ can be rewritten 
in the equivalent form:
\begin{align}
& \lambda \int_{-\infty}^{+\infty} e^{2t} \Theta_h(t)^2 dt [1 + \mathcal{O}(b^{-2ap},b^{-2a(1-p)})] \notag \\
& -   2 b^{-2p} \int_{-\infty}^{+\infty} e^{4t} \Theta_h(t)^2 dt [1 + \mathcal{O}(b^{-2ap},b^{-2a(1-2p)})] + 
\mathcal{O}((|\lambda| b^{-2p} + b^{-4p})^{2p+1} b^{2p(1+a)}) \notag \\
&= 	 -2 p^{-2} \alpha_p^{-2/p} \frac{\Gamma\left(\frac{p+1}{2p}-\frac{\lambda}{4}\right)\Gamma\left(-\frac{1}{p}\right)}{\Gamma\left(\frac{p-1}{2p}-\frac{\lambda}{4}\right)\Gamma\left(\frac{1}{p}\right)} b^{-2(1-p)} [ 1 + \mathcal{O}(b^{-2p(1-a)})]. \label{eq-2-new}
\end{align}
If $p \neq \frac{1}{n}$ for every $n \in \mathbb{N}$, then $\frac{p-1}{2p},-\frac{1}{p} \neq -m$ for every $m \in \mathbb{N}$ so that the arguments of the Gamma functions at $\lambda = 0$ are away from their pole singularities. Hence, all terms of the nonlinear equation (\ref{eq-2-new}) are $C^1$ functions of $\lambda$ at $\lambda = 0$. 
For $p \in (0,\frac{1}{2})$, $b^{-2(1-p)} \ll b^{-2p}$ for sufficiently large $b$. Since the leading-order part of the nonlinear equation (\ref{eq-2-new}) 
is linear in $\lambda$ and suggests the solution $\lambda = \mathcal{O}(b^{-2p})$, we have by an application of the implicit function 
theorem the existence of a $C^1$ function $\lambda = \lambda(b)$ for sufficiently large $b$ which is given asymptotically by 	
\begin{equation}
\label{eq:lamb_c1}
\lambda(b) = 2 \frac{\int_{-\infty}^{+\infty}e^{4t}\Theta_h(t)^2dt}{\int_{-\infty}^{+\infty}e^{2t}\Theta_h(t)^2dt}b^{-2p} + \mathcal{O}(b^{-2(1-p)},b^{-2p(1+a)},b^{-2(p+a(1-2p))},b^{-2p(4p+1-a)}).
\end{equation}
The integrals in (\ref{eq:lamb_c1}) can be computed by using 
the explicit expression for $\Theta_h$ given by 
\eqref{eq:Theta_alg_sol} with $t_0=0$. Using the substitution  $s=(1+\alpha_pe^{2t})^{-1}$ we express the integrals in terms of the 
Beta function 
$$
B(z_1,z_2) := \int_0^1 s^{z_1-1} (1-s)^{z_2-1} ds = 
\frac{\Gamma(z_1)\Gamma(z_2)}{\Gamma(z_1+z_2)}, \quad z_1,z_2>0
$$
and obtain		
\begin{align*}
\int_{-\infty}^{+\infty}e^{2t}\Theta_h(t)^2dt &= \frac{1}{2\alpha_p^{1+1/p}}\int_0^1s^{\frac{1}{p}-2}(1-s)^{\frac{1}{p}}ds = \frac{\Gamma\left(\frac{1}{p}-1\right)\Gamma\left(\frac{1}{p}+1\right)}{2\alpha_p^{1+1/p} \Gamma\left(\frac{2}{p}\right)}, \\
\int_{-\infty}^{+\infty}e^{4t}\Theta_h(t)^2dt &= \frac{1}{2\alpha_p^{2+1/p}}\int_0^1 s^{\frac{1}{p}-3}(1-s)^{\frac{1}{p}+1}ds = \frac{\Gamma\left(\frac{1}{p}-2\right)\Gamma\left(\frac{1}{p}+2\right)}{2\alpha_p^{2+1/p} \Gamma\left(\frac{2}{p}\right)}.
\end{align*}
Substituting these expressions into \eqref{eq:lamb_c1} yields 
the final formula for $p \in (0,\frac{1}{2})$ and $p \neq \frac{1}{n}$ for every $n \in \mathbb{N}$:
\begin{equation}
\label{eq:lamb_c2}
\lambda(b) = \frac{2 (1+p)}{\alpha_p (1-2p)} b^{-2p} + \mathcal{O}(b^{-2(1-p)},b^{-2p(1+a)},b^{-2(p+a(1-2p))},b^{-2p(4p+1-a)}),
\end{equation}
where we have used the property $\Gamma(z+1) = z \Gamma(z)$. \\

If $p = \frac{1}{n}$ for some $n \in \mathbb{N}\backslash \{1,2\}$,  we use (\ref{eq:Psi_b_match}) and 
(\ref{eq:Psi_c_t_c1}) in (\ref{intersection}), equal the coefficients for the 
$e^{\frac{T}{p}}$ terms, and substitute
the expression (\ref{c-solution}) for $c = c(\lambda,b)$. This yields the nonlinear equation for $\lambda$:
\begin{align}
& \frac{1}{2} p^2 \alpha_p^{1/p} b^a \big[ 
\lambda b^{-2p} \int_{-\infty}^{T+ap\log b} e^{2t} \Theta_h'(t) \Theta_h(t) dt [1 + \mathcal{O}(b^{-2ap})] \notag \\
& -   b^{-4p} \int_{-\infty}^{T+ap\log b} e^{4t} \Theta_h'(t) \Theta_h(t) dt [1 + \mathcal{O}(b^{-2ap})] + 
\mathcal{O}((|\lambda| b^{-2p} + b^{-4p})^{2p+1} b^{2ap}) \big] \notag \\
&= 	  \frac{2 (-1)^{n+1} \Gamma\left( \frac{1+n}{2} - \frac{\lambda}{4} \right)}{\alpha_p^n n! (n-1)!   \Gamma\left(\frac{1-n}{2}-\frac{\lambda}{4}\right)} b^{-2+a} [(T - (1-a) p \log b) [ 1 + \mathcal{O}(b^{-2p(1-a)})] + \mathcal{O}(1) ]. \label{eq-5}
\end{align}
After dividing it by $b^{a-2p}$, this equation can be rewritten in the form 
(\ref{eq-2-new}), where the right-hand side has the order of 
$$
\frac{\log b \; b^{-2(1-p)}}{\Gamma(\frac{1-n}{2} - \frac{\lambda}{4})},
$$ 
which is much smaller than the leading-order term of the order of $\mathcal{O}(b^{-2p})$ for $p \in (0,\frac{1}{2})$. For even $n$, the final formula (\ref{eq:lamb_c1}) for $\lambda(b)$ is modified as follows:
\begin{equation*}
\lambda(b) = \frac{2 (1+p)}{\alpha_p (1-2p)} b^{-2p} + \mathcal{O}(\log b b^{-2(1-p)},b^{-2p(1+a)},b^{-2(p+a(1-2p))},b^{-2p(4p+1-a)}).
\end{equation*}
For odd $n$, we also have $\Gamma(\frac{1-n}{2}) = \infty$. Since 
$$
\Gamma(z) = \frac{(-1)^n}{n! (z+n)} + \mathcal{O}(1) \quad \mbox{\rm as} \;\; z \to -n
$$
and $\lambda(b) = \mathcal{O}(b^{-2p})$, we have 
$$
\frac{1}{
\Gamma(\frac{1-n}{2} - \frac{\lambda}{4})} = \mathcal{O}(\lambda) = \mathcal{O}(b^{-2p}),
$$
which modifies the final formula (\ref{eq:lamb_c1}) for $\lambda(b)$ 
according to 
\begin{equation*}
\lambda(b) = \frac{2 (1+p)}{\alpha_p (1-2p)} b^{-2p} + \mathcal{O}(\log b b^{-2},b^{-2p(1+a)},b^{-2(p+a(1-2p))},b^{-2p(4p+1-a)}).
\end{equation*}
In both formulas for $\lambda = \lambda(b)$, we have $p = \frac{1}{n}$ with either even or odd $n \in \mathbb{N}\backslash \{1,2\}$. In all cases, $\lambda(b) > 0$ for suffiently large values of $b$.\\

{\bf Case $p \in (\frac{1}{2},1)$}. Since $p \neq \frac{1}{n}$ for every $n \in \mathbb{N}$ if $p \in (\frac{1}{2},1)$, the nonlinear equation (\ref{eq-2}) can be used. However, the integral $\int_{-\infty}^{T + ap \log b} e^{4t} \Theta_h(t)^2 dt$ diverges exponentially in the upper limit since $\Theta_h(t)^2 = \mathcal{O}(e^{-\frac{2t}{p}})$ as $t \to +\infty$. Consequently, there is a positive constant $C_{T,a}$ such that for all $b \geq b_{T,a}$, we have 
\begin{equation*}
\left| \int_{-\infty}^{T+ap\log b}e^{4t}\Theta_h'(t)\Theta_h(t)dt \right| \leq  C_{T,a} b^{2a(2p-1)}.
\end{equation*}
The nonlinear equation (\ref{eq-2}) can be rewritten in the equivalent form:
\begin{align}
& \lambda \int_{-\infty}^{+\infty} e^{2t} \Theta_h(t)^2 dt [1 + \mathcal{O}(b^{-2ap},b^{-2a(1-p)})] \notag \\
& + b^{-2p} \int_{-\infty}^{T+ap \log b} e^{4t} \Theta_h'(t) \Theta_h(t) dt [1 + \mathcal{O}(b^{-2ap})] + 
\mathcal{O}((|\lambda| b^{-2p} + b^{-4p(1-a)})^2 b^{2p(1+a)}) \notag \\
&= 	 -2 p^{-2} \alpha_p^{-2/p} \frac{\Gamma\left(\frac{p+1}{2p}-\frac{\lambda}{4}\right)\Gamma\left(-\frac{1}{p}\right)}{\Gamma\left(\frac{p-1}{2p}-\frac{\lambda}{4}\right)\Gamma\left(\frac{1}{p}\right)} b^{-2(1-p)} [ 1 + \mathcal{O}(b^{-2p(1-a)})], \label{eq-3}
\end{align}
where the second term on the left-hand side is of the order of 
$b^{-2p +2 a (2p-1)} \to 0$ as $b \to \infty$ since $-2p+2a(2p-1)<0$ for $a<\frac{p}{2p-1}$, which is satisfied automatically, since $\frac{p}{2p-1} > 1$ for $p \in (\frac{1}{2},1)$. Moreover, since $b^{-2p + 2a(2p-1)}\ll b^{-2(1-p)}$ for $p\in\left(\frac{1}{2},1\right)$ and $a \in (0,1)$, the right-hand side dominates in the nonlinear equation (\ref{eq-3}). Solving the nonlinear equation (\ref{eq-3}) by an application of the implicit function theorem, we have the existence of a $C^1$ function $\lambda = \lambda(b)$ for sufficiently large $b$ which is given asymptotically by 
\begin{align}
\notag
\lambda(b) =& -\frac{4\alpha_p^{1-1/p}\Gamma\left(\frac{p+1}{2p}\right)\Gamma\left(-\frac{1}{p}\right)\Gamma\left(\frac{2}{p}\right)}{p^2 \Gamma\left(\frac{p-1}{2p}\right)\Gamma\left(\frac{1}{p}\right)\Gamma\left(\frac{1}{p}-1\right) \Gamma\left(\frac{1}{p}+1\right)} b^{-2(1-p)} \\
& \qquad +  \mathcal{O}(b^{-2p+2a(2p-1)},b^{-2(1-p)-2ap},b^{-2(1+a)(1-p)},b^{-2(1-pa)},b^{-2p(3-5a)}).\label{eq:lamb_c3}
\end{align}
Since $-\frac{1}{p}\in(-2,-1)$ and $\frac{p-1}{2p}\in(-\frac{1}{2},0)$ if $p \in (\frac{1}{2},1)$, we have $\Gamma\left(-\frac{1}{p}\right)>0$ and  $\Gamma\left(\frac{p-1}{2p}\right)<0$. Hence, $\lambda(b) > 0$ for sufficiently large values of $b$.\\

{\bf Case $p  = \frac{1}{2}$}. This case corresponds to $n = 2$ in the nonlinear equation (\ref{eq-5}), which we can rewrite in the equivalent form:
\begin{align}
&  \lambda \int_{-\infty}^{+\infty} e^{2t} \Theta_h(t)^2 dt [1 + \mathcal{O}(b^{-a})] \notag \\
& +  b^{-1} \int_{-\infty}^{T+ap\log b} e^{4t} \Theta_h'(t) \Theta_h(t) dt [1 + \mathcal{O}(b^{-a})] + 
\mathcal{O}((|\lambda| b^{-1} + b^{-2} \log b)^{2} b^{a}) \notag \\
&= 	  \frac{8 \Gamma\left( \frac{3}{2} - \frac{\lambda}{4} \right)}{\alpha_p^4   \Gamma\left(-\frac{1}{2}-\frac{\lambda}{4}\right)} b^{-1} [(T - (1-a) p \log b) [ 1 + \mathcal{O}(b^{-(1-a)})] + \mathcal{O}(1) ], \label{eq-6}
\end{align}
where $\alpha_{p = \frac{1}{2}} = \frac{1}{4!}$.
The second integral $\int_{-\infty}^{T+ap\log b} e^{4t} \Theta_h'(t)\Theta_h(t) dt$ diverges linearly in the upper limit since $\Theta_h(t)^2 = \mathcal{O}(e^{-4t})$ as $t \to +\infty$. The exact computations 
with the help of the explicit formula (\ref{eq:Theta_alg_sol}) yield the 
following asymptotic expression:
\begin{align*}
\int_{-\infty}^{T+ap \log b}e^{4t}\Theta_h'(t)\Theta_h(t)dt & = \frac{1}{2} e^{4t} \Theta_h(t)^2 \biggr|^{t = T+ap \log b}_{t \to -\infty}  - 2\int_{-\infty}^{T+ap \log b}e^{4t}\Theta_h(t)^2 dt\\
&= \frac{1}{2 \alpha_p^4} - \frac{1}{\alpha_p^4} \left[ 
2(T + a p \log b) + \log(\alpha_p) - \frac{11}{6} + \mathcal{O}(b^{-2ap}) \right] \\
&= -\frac{2}{\alpha_p^4} (T + ap \log b) + \frac{7}{3 \alpha_p^4} - \frac{\log(\alpha_p)}{\alpha_p^4} + \mathcal{O}(b^{-2ap}).
\end{align*}
On the other hand, we use (\ref{Gamma-function}) and obtain
$$
\frac{8 \Gamma\left(\frac{3}{2}\right)}{\Gamma\left(-\frac{1}{2}\right)} = 
-\frac{8}{\pi}  \Gamma^2\left(\frac{3}{2}\right) = -2,
$$
so that the leading-order terms of the nonlinear equation (\ref{eq-6}) can be collected together as 
\begin{align*}
&  \lambda \int_{-\infty}^{+\infty} e^{2t} \Theta_h(t)^2 dt [1 + \mathcal{O}(b^{-a})] +  
\mathcal{O}(b^{-1},\log b \; b^{-1-a}) + 
\mathcal{O}((|\lambda| b^{-1} + b^{-2} \log b)^{2} b^{a})  \\
&= 	  \frac{1}{\alpha_p^4} \log b \; b^{-1} [1 + \mathcal{O}(\lambda)] [ 1 + \mathcal{O}(b^{-(1-a)})] + \mathcal{O}(b^{-1}).
\end{align*}
By using the implicit function theorem, we have the existence of a $C^1$ function $\lambda = \lambda(b)$ for sufficiently large $b$ 
which is given asymptotically by 
		\begin{equation}
			\lambda(b) = 144 \log b \; b^{-1} + \mathcal{O}(b^{-1},\log b \; b^{-1-a},(\log b)^2 \; b^{-2}),
			\label{eq:lamb_c4}
		\end{equation}
where we have used $p = \frac{1}{2}$, $\alpha_{p=\frac{1}{2}} = \frac{1}{4!}$, 
and 
$$
\int_{-\infty}^{+\infty} e^{2t} \Theta_h(t)^2 dt = \frac{1}{6 \alpha_p^3}.
$$
Hence, $\lambda(b)>0$ for sufficiently large values of $b$.\\

Theorem \ref{theorem-main} is proven. For details in Remark \ref{remark-fail}, 
we give the following computations.\\

{\bf Case $p  = 1$}. This case corresponds to $n = 1$ in the nonlinear equation (\ref{eq-5}), which we can rewrite in the equivalent form:
	\begin{align*}
&  \lambda \int_{-\infty}^{T + a \log b} e^{2t} \Theta_h'(t) \Theta_h(t) dt [1 + \mathcal{O}(b^{-2a})] \\
& - b^{-2} \int_{-\infty}^{T+a\log b} e^{4t} \Theta_h'(t)\Theta_h(t)dt [1 + \mathcal{O}(b^{-2a})] + 
\mathcal{O}((|\lambda| b^{-2} \log b + b^{-4(1-a)})^{2} b^{2a})  \\	
	&= 	  \frac{4 \Gamma\left( 1 - \frac{\lambda}{4} \right)}{\alpha_p^2   \Gamma\left(-\frac{\lambda}{4}\right)} [(T - (1-a) \log b) [ 1 + \mathcal{O}(b^{-2(1-a)})] + \mathcal{O}(1) ],
	\end{align*}
where $\alpha_{p=1} = \frac{1}{8}$. Since 
$$
\Gamma(z) = \frac{1}{z} + \mathcal{O}(1) \quad \mbox{\rm as} \;\; z \to 0
$$
and 
\begin{align*}
\int_{-\infty}^{T+a \log b}e^{2t}\Theta_h'(t)\Theta_h(t)dt 
= -\frac{1}{\alpha_p^2} (T + a \log b - 1) - \frac{1}{2 \alpha_p^2} \log(\alpha_p) + \mathcal{O}(b^{-2a}),
\end{align*}
the leading-order terms contain only $\lambda \log b$, which are not balanced by the terms of the order of $\mathcal{O}(1)$ to get the asymptotic balance $\lambda = \mathcal{O}((\log b)^{-1})$ according to Remark \ref{remark-fail}. This failure of the shooting method is due to only one exponential term that appears in (\ref{eq:Psi_c_t_c1}) for $n = 1$ and $\lambda = 0$. The way to handle the asymptotic balance is to obtain the second exponential terms from 
the higher-order (nonlinear) terms of the expansion for $\Psi_c(t)$ beyond the leading order. However, this adds complexity to the shooting method beyond the scopes of this work.

\end{document}